\documentclass[12pt,reqno,a4paper,twoside]{article}
\usepackage{geometry}
\usepackage{amsmath,amsthm,amstext,amscd,amssymb,euscript,url}
\usepackage{mathrsfs}
\usepackage{mathtools}
\usepackage{amscd}
\usepackage{verbatim}
\usepackage{amscd}
\usepackage{verbatim}
\usepackage{enumerate}
\usepackage{tikz}
\usetikzlibrary{matrix}
\usetikzlibrary{arrows,positioning}
\usepackage{setspace}
\usepackage{hyperref}
\usepackage{scrextend}
\usepackage{sectsty}
\usepackage{verbatim}
\usepackage{textcomp}
\usepackage{epsfig}
\usepackage[inline]{enumitem}
\usepackage{todonotes}
\usepackage{calrsfs}
\usepackage[absolute,overlay]{textpos}
\usepackage{graphicx}
\usepackage{subfig}
\usepackage{color}

\numberwithin{equation}{section}
\newtheorem{theorem}{Theorem}[section]
\newtheorem{corollary}[theorem]{Corollary}

\newtheorem{proposition}[theorem]{Proposition}
\newtheorem{lemma}[theorem]{Lemma}
\newtheorem{remark}[theorem]{Remark}
\newtheorem{definition}[theorem]{Definition}
\theoremstyle{plain}

\def\N{{\mathbb N}}
\def\Z{{\mathbb Z}}

\def\R{{\mathbb R}}

\newcommand{\E}{{\mathbb E}}
\renewcommand{\P}{{\mathbb P}}

\newcommand{\eft}{\text{left}}
\newcommand{\ight}{\text{right}}

\newcommand{\res}{\text{\normalfont res}}

\newcommand{\abs}[1]{\ensuremath{\left\lvert #1 \right\rvert}} 

\DeclareRobustCommand{\stirling}{\genfrac\{\}{0pt}{}}

\newmuskip\pFqmuskip
\newcommand*\pFq[6][8]{%
	\begingroup 
	\pFqmuskip=#1mu\relax
	\mathcode`\,=\string"8000
	\begingroup\lccode`\~=`\,
	\lowercase{\endgroup\let~}\pFqcomma
	{}_{#2}F_{#3}{\left[\genfrac..{0pt}{}{#4}{#5};#6\right]}%
	\endgroup
}
\newcommand{\pFqcomma}{\mskip\pFqmuskip}

\newcommand{\IRW}{\text{\normalfont IRW}}
\newcommand{\SIP}{\text{\normalfont SIP}}

\newcommand{\SEP}{\text{\normalfont SEP}}

\newcommand{\cond}{\omega}
\newcommand{\scale}{\theta}
\newcommand{\Scale}{\varTheta}
\newcommand{\cl}{{c\ell}}
\newcommand{\oo}{{or}}
\newcommand{\Bin}{\text{\normalfont Binomial}}
\newcommand{\Poi}{\text{\normalfont Poisson}}

\newcommand{\NegBin}{\text{\normalfont Negative-Binomial}}

\renewcommand{\Pr}{\mathbb{P}}

\newcommand{\bulk}{\text{\normalfont bulk}}
\newcommand{\bdd}{{L,R}}

\newcommand{\dd}{\text{\rm d}}
\newcommand{\1}{{\boldsymbol 1}}
\newcommand{\bubu}{b}

\allowdisplaybreaks
\usepackage[inline]{enumitem}

\usepackage{sectsty}

\sectionfont{\fontsize{13}{13}\selectfont}
\subsectionfont{\fontsize{11}{11}\selectfont}
\subsubsectionfont{\fontsize{10}{10}\selectfont}

\newcommand{\footremember}[2]{%
	\footnote{#2}
	\newcounter{#1}
	\setcounter{#1}{\value{footnote}}%
}

\makeatletter
\let\@fnsymbol\@arabic
\makeatother

\title{{\bf \Large 
		Orthogonal polynomial duality of \\
		boundary driven  particle systems and non-equilibrium correlations}}

\author{Simone Floreani\footremember{Delft}{Delft Institute of Applied Mathematics, TU Delft,
		Delft, The Netherlands. \href{mailto:s.floreani@tudelft.nl}{s.floreani@tudelft.nl}.}
	\and
		Frank Redig\footremember{Delft2}{Delft Institute of Applied Mathematics, TU Delft,
			Delft, The Netherlands. \href{mailto:f.h.j.redig@tudelft.nl}{f.h.j.redig@tudelft.nl}.}
	\and
Federico Sau\footremember{ISTAustria}{Institute of Science and Technology, IST Austria, Klosterneuburg, Austria. \href{mailto:federico.sau@ist.ac.at}{federico.sau@ist.ac.at}.}}

\providecommand{\keywords}[1]
{
	\small	
	\textbf{{Keywords ---}} #1
}
\begin{document}
\maketitle

	\begin{abstract}
		We consider symmetric  partial  exclusion and 	 inclusion processes in a general graph in contact with reservoirs, where we allow both
		for edge disorder and well-chosen site disorder.
		We extend the classical dualities to this  context and then  we derive new orthogonal polynomial dualities. From the classical dualities, we derive the uniqueness of the non-equilibrium steady state and obtain correlation  inequalities.
		Starting from the orthogonal polynomial dualities, we show universal properties of $n$-point correlation functions in the non-equilibrium steady state for  systems with at most   two different reservoir parameters, such as a chain with  reservoirs at left and right ends. 
	\end{abstract}
\keywords{Interacting particle systems; boundary driven systems; duality; orthogonal polynomial duality; non-equilibrium stationary measure; non-equilibrium stationary correlations; symmetric exclusion process; symmetric inclusion process.}

\section{Introduction}
Exactly solvable models have played an important role in the understanding of fundamental properties of non-equilibrium steady states such as the presence of long-range correlations and the non-locality of large deviation free energies \cite{derrida_exact_1993-1, derrida_entropy_2007, bertini_macroscopic_2015,gilbert_heat_2017}.
An important class of particle systems which is slightly broader than exactly solvable models are the models which satisfy self-duality or, more generally, duality properties. Such systems when coupled to appropriate reservoirs are dual to systems where the reservoirs are replaced by absorbing boundaries, and the computation of $n$-point correlation functions in the original system reduces to the computation of absorption probabilities in a dual system with $n$ particles. Even when these absorption probabilities cannot be obtained in closed form, e.g.\ when Bethe ansatz is not available, still the connection between the non-equilibrium system coupled to reservoirs and the absorbing dual turns out to be very useful to obtain macroscopic properties such as the hydrodynamic limit, fluctuations, mixing and propagation of chaos and local equilibrium (see e.g.\ \cite{goncalves_non-equilibrium_2019, landim_stationary_2006, nandori_local_2016, franceschini2020symmetric, gantert2020mixing}).

In recent works (self-)duality with orthogonal polynomials has been studied in several particle systems including generalized symmetric exclusion processes ($\SEP$), symmetric inclusion process ($\SIP$) and associated diffusion processes such as the Brownian momentum process. Orthogonal polynomials in the occupation  number variables are a natural extension of the higher order correlation functions studied in
$\SEP$ in \cite{derrida_entropy_2007}.
Orthogonal polynomial duality is very useful in the study of fluctuation fields \cite{ayala2020higher, chen2020higher}, identifies a set of functions with positive time dependent correlations and is useful in the study of speed of relaxation to equilibrium \cite{carinci_book}.
So far, orthogonal polynomial duality has not been obtained in the context of boundary driven systems.

In this paper we start extending the classical dualities from \cite{carinci_duality_2013-1} for a generalized class of boundary driven systems, where we allow both for edge disorder and well-chosen site disorder. 
We then use a symmetry of the dual absorbing system in order to derive duality with orthogonal polynomials for these systems.

More precisely, we consider three classes of interacting particle systems: partial  symmetric exclusion \cite{floreani2019hydrodynamics} where we allow edge-dependent conductances and a site-varying maximal occupancy,  symmetric inclusion where we allow edge-dependent conductances and a site-varying ``attraction parameter", and independent walkers.
We couple these systems to two reservoirs, with reservoir parameters $\scale_L$ and $\scale_R$. The precise meaning of the reservoir parameters $\scale_L$ and $\scale_R$ will be explained in detail later; for the moment one can think of them -- roughly -- as being proportional to the densities of left and right reservoirs, respectively. Moreover, the bulk system can be defined on any graph. Hence, our setting  includes the standard one of a chain coupled to reservoirs at left and right ends, but it is in no way restricted to that setting. The only important geometrical requirement is the presence of precisely two reservoirs. When $\scale_L=\scale_R=\scale$ the system is in equilibrium, with a unique reversible product measure $\mu_\scale$. When $\scale_L\not=\scale_R$ the system evolves towards a unique non-equilibrium stationary measure $\mu_{\scale_L, \scale_R}$. At stationarity, by means of classical dualities  with a dual system that has two absorbing sites, corresponding to the reservoirs in the original system, we obtain correlation inequalities, thereby extending and strengthening those from \cite{giardina_correlation_2010}.
In particular, the dual particle system dynamics does not depend on the reservoir parameters $\scale_L$ and $\scale_R$.

Next, for the same pair of boundary driven and purely absorbing systems, we introduce   orthogonal polynomial dualities.     The orthogonal duality functions are in product form and the factors associated to the bulk sites are the same orthogonal polynomials as those appearing for the same particle systems not coupled to reservoirs (see e.g.\ \cite{franceschini_stochastic_2017, redig_factorized_2018}), while the remaining factors corresponding to the absorbing sites have a form depending on the reservoir parameters.
The orthogonal polynomials carry themselves a parameter $\scale$ which corresponds to the equilibrium reversible product measure $\mu_\scale$ w.r.t.\ which they are orthogonal.

We then give various applications of these orthogonal polynomial dualities to properties of correlation functions in the non-equilibrium stationary measure $\mu_{\scale_L, \scale_R}$. First we prove that the correlations of order $n$ of the occupation variables at
different locations $x_1, \ldots, x_n$, as well as the cumulants of order $n$, are of the form
$(\scale_L-\scale_R)^n$ multiplied by a universal function $\psi$ which depends only on $x_1,\ldots, x_n$ and	 the dual particle system dynamics, thus, not depending on $\scale_L$ and $\scale_R$.
We prove, in fact, a stronger result, namely that whenever the system is started from a local equilibrium product measure, then,
at any later time $t>0$, the $n$-point correlations are of the form $(\scale_L-\scale_R)^n$ multiplied by a universal function $\psi_t$ which, again,   does not depend  on the reservoir parameters $\scale_L$ and $\scale_R$, but only on the dual system dynamics. 

Finally, we relate the joint moment generating function of the occupation variables to an expectation in the absorbing dual. Despite the fact that this quantity can in general not be obtained in analytic form, the relation is useful, both from the point of view of simulations, as well as from the point of view of computing macroscopic limits such as  density fluctuation fields and  large deviations of the density profile.

\subsection{Summary of main results, related works  and organization of the paper}
As a conclusion of this introduction,   we summarize more schematically here, for the convenience of the reader,    our main contributions in relation to previous works and how we organize the rest of the paper.
 
 We introduce a  class of  boundary driven particle systems in a general inhomogeneous framework -- generalizing, in particular, those considered in, e.g., \cite{carinci_duality_2013-1, derrida_entropy_2007} -- showing that classical dualities may be extended beyond homogeneous systems. As a first main result, employing these classical dualities, we  show that \emph{correlations} of interacting systems are \emph{monotone  in time} when starting from suitable local equilibrium product measures. As a consequence, we deduce a family of correlation inequalities,  improving on those established for homogeneous  symmetric exclusion  and inclusion processes in, e.g., \cite{giardina_duality_2007, giardina_correlation_2010, liggett_interacting_2005-1}. 	
 
As a second main result, in our  context of boundary driven systems, we derive the \emph{orthogonal polynomial dualities}, 
 previously studied in \cite{franceschini2020symmetric, redig_factorized_2018, carinci2019orthogonal, groenevelt_orthogonal_2018} for closed systems. To this purpose, we develop a 	new method, which is of independent interest and is based on the relation between orthogonal and  classical duality functions. 
 For further details, we refer to the discussion following Theorem \ref{proposition:orthogonal_duality}.
 
As a third main result, by suitably tilting these orthogonal dualities, we show that \emph{$n$-point non-equilibrium stationary correlations} and \emph{cumulants} exhibit a \emph{universal factorized  structure},  one factor consisting in a simple expression in the reservoir parameters and the other factor depending only on the underlying geometry of the system. This result holds for both boundary driven  exclusion and inclusion processes in presence of edge and site disorder. In particular, for these more general systems, this result recovers   the same structure previously obtained for the boundary driven one-dimensional $\SEP$ in \cite{derrida_entropy_2007} by means of the explicit knowledge via matrix formulation of the non-equilibrium steady state.

The rest of our paper is organized as follows. In Section \ref{section models} we define the  boundary driven particle systems and their dual absorbing processes as well as introducing the classical duality functions. In Section \ref{section:measures} we study properties and correlation inequalities for the equilibrium and non-equilibrium stationary measures. In Section \ref{section:orthogonal_duality} we derive orthogonal duality functions between the boundary driven and the absorbing systems. In Section \ref{section:n-point} we obtain the aforementioned universal expression for the higher order correlations in the non-equilibrium steady state. In the same section, the same structure is recovered for more general correlations at finite times when started from a local equilibrium product measure. Section \ref{section:exponential_generating_functions} is devoted to a relation between weighted exponential generating functions of the occupation variables at stationarity and the correlation functions obtained in the previous section. In conclusion, Appendix \ref{appendix:existence_uniqueness} contains part of the proof of Theorem \ref{theorem:existence_uniqueness} in Section \ref{section:measures}.

\section{Setting} \label{section models}
In this section, we start by  introducing the common geometry and the disorder on which the particle  dynamics takes place. Then, we couple this \textquotedblleft bulk\textquotedblright\ system to two reservoirs at possibly different densities.

\subsection{Boundary driven particle systems}\label{section:particle_systems}

We consider three particle systems with either an exclusion, inclusion or no interaction. All these systems will evolve on a set of sites $V = \{1,\ldots, N \}$ ($N \in \N$) and the  rate  of particle exchanges between two sites $x$ and $y \in V$ will be proportional to some given (symmetric) conductance $\cond_{\{x,y\}} \in [0,\infty)$. Sites $x$ and $y \in V$ for which $\cond_{\{x,y\}} \neq 0$ will be considered as connected, indicated by $x \sim y$. In what follows, we will assume that $\cond_{\{x,x\}}=0$ for all $x \in V$ and that the induced graph $(V,\sim)$ is connected. We will further attach to each site $x \in V$ a value $\alpha_x \in \N$. While the conductances $\boldsymbol \cond=\{\cond_{\{x,y\}}: x, y \in V \}$ represent the bond disorder, the collection $\boldsymbol \alpha=\{\alpha_x: x \in V \}$ stands for the  site disorder. This disorder may be thought, e.g., as a realization of a random environment (see, e.g., \cite{nandori_local_2016, floreani2019hydrodynamics}); however, our work is not focusing on homogenization properties arising from the randomness of the disorder. Instead, we consider the disorder as deterministic and parameterizing the model all throughout the paper.

The set $V$ endowed with the disorder $(\boldsymbol \cond, \boldsymbol \alpha)$ is referred to as \emph{bulk} of the system. This bulk is in contact with a left and a right reservoir through respectively site $1$ and site  $N \in V$. Particle exchanges between the bulk sites and the reservoirs is tuned by a set of non-negative parameters $\cond_L$, $\cond_R$, $\scale_L$, $\scale_R$, $\alpha_L$ and $\alpha_R$ as explained in the paragraph below.

\subsubsection{Particle dynamics} In this setting, for each choice of the parameter $\sigma \in \{-1,0,1\}$, we introduce a boundary driven particle system $\{\eta_t: t\geq 0 \}$ as a Markov process with  $\mathcal X$, given by
\begin{align*}
\mathcal X = \begin{dcases}
\prod_{x \in V} \{0,\ldots, \alpha_x \} &\text{if } \sigma = -1\\[.1cm]
\prod_{x \in V} \{0,1,\ldots \}=
\N_0^{V}&\text{otherwise}\ ,	
\end{dcases}
\end{align*} denoting the configuration space, with $\eta \in \mathcal X$ standing for  a particle configuration and with $\eta(x)$ indicating the number of particles at site $x \in V$ for the configuration $\eta \in \mathcal X$. The particle dynamics is described by the infinitesimal generator  $\mathcal L$, whose action on bounded functions $f : \mathcal X \to \R$ reads as follows:
\begin{align}\label{eq:generator_full}
\mathcal Lf(\eta) = \mathcal L^\bulk f(\eta) + \mathcal L^\bdd f (\eta)\ .
\end{align}
In the above expression, the generator $\mathcal L^\bulk$ describes the bulk part of the dynamics and is given by
\begin{align}\label{eq:generator_bulk}
\mathcal L^\bulk f(\eta) = \sum_{x \sim y} \cond_{\{x,y\}}\,  \mathcal L_{\{x,y\}}f(\eta)
\end{align}
where the summation above runs over the unordered pairs of sites and with the single-bond generator $\mathcal L_{\{x,y\}}$ given by
\begin{align*}
\mathcal L_{\{x,y\}} f(\eta) =&\ \eta(x)\, (\alpha_y+\sigma \eta(y))\, (f(\eta^{x,y})-f(\eta))\\
+&\ \eta(y)\, (\alpha_x+\sigma \eta(x))\, (f(\eta^{y,x})-f(\eta))\ ,
\end{align*}
where $\eta^{x,y} = \eta- \delta_x+\delta_y \in \mathcal X$, i.e.\  the configuration in which a particle (if any) has been removed from $x \in V$ and placed at $y \in V$. The boundary part of the dynamics is described by the generator $\mathcal L^\bdd$ in \eqref{eq:generator_full} as follows:
\begin{align}\label{eq:generator_boundary}
\mathcal L^\bdd f(\eta) =&\ \cond_L\, \mathcal L_L f(\eta) + \cond_R\, \mathcal L_R f(\eta)\ ,
\end{align}
with
\begin{align}\label{eq:L_L}
\nonumber
\mathcal L_L f(\eta) =&\  \eta(1)\, (\alpha_L+\sigma \alpha_L \scale_L)\,	 (f(\eta^{1,-})-f(\eta))\\
+&\ \alpha_L \scale_L\, (\alpha_1+\sigma \eta(1))\, (f(\eta^{1,+})-f(\eta))	
\end{align}
and
\begin{align}\label{eq:L_R}
\nonumber
\mathcal L_R f(\eta) =&\  \eta(N)\, (\alpha_R+\sigma \alpha_R \scale_R) \left(f(\eta^{N,-})-f(\eta)\right)\\
+&\  \alpha_R \scale_R \left(\alpha_N+\sigma \eta(N))\, (f(\eta^{N,+})-f(\eta)\right) \ ,	
\end{align}
where $\eta^{x,-} \in \mathcal X$, resp.\ $\eta^{x,+} \in \mathcal X$, denotes the configuration obtained from $\eta$ by removing, resp.\ adding, a particle from, resp.\ to,   site $x \in V$. In the above dynamics, creation and annihilation of particles occurs at sites $x = 1$ and $x = N$ due to the interaction with a reservoir.

We note that, depending on the choice of the value  $\sigma \in \{-1,0,1\}$ in the definition of the generator $\mathcal L$ in \eqref{eq:generator_full}, we recover either the \emph{symmetric partial exclusion process} ($\SEP$) for $\sigma = -1$, a system of \emph{independent random walkers} ($\IRW$) for $\sigma = 0$ or the \emph{symmetric inclusion process} ($\SIP$) for $\sigma = 1$ in contact with left and right reservoirs and in presence of disorder.

\begin{figure}[h]
	\centering
	\includegraphics[width=1.0\linewidth]{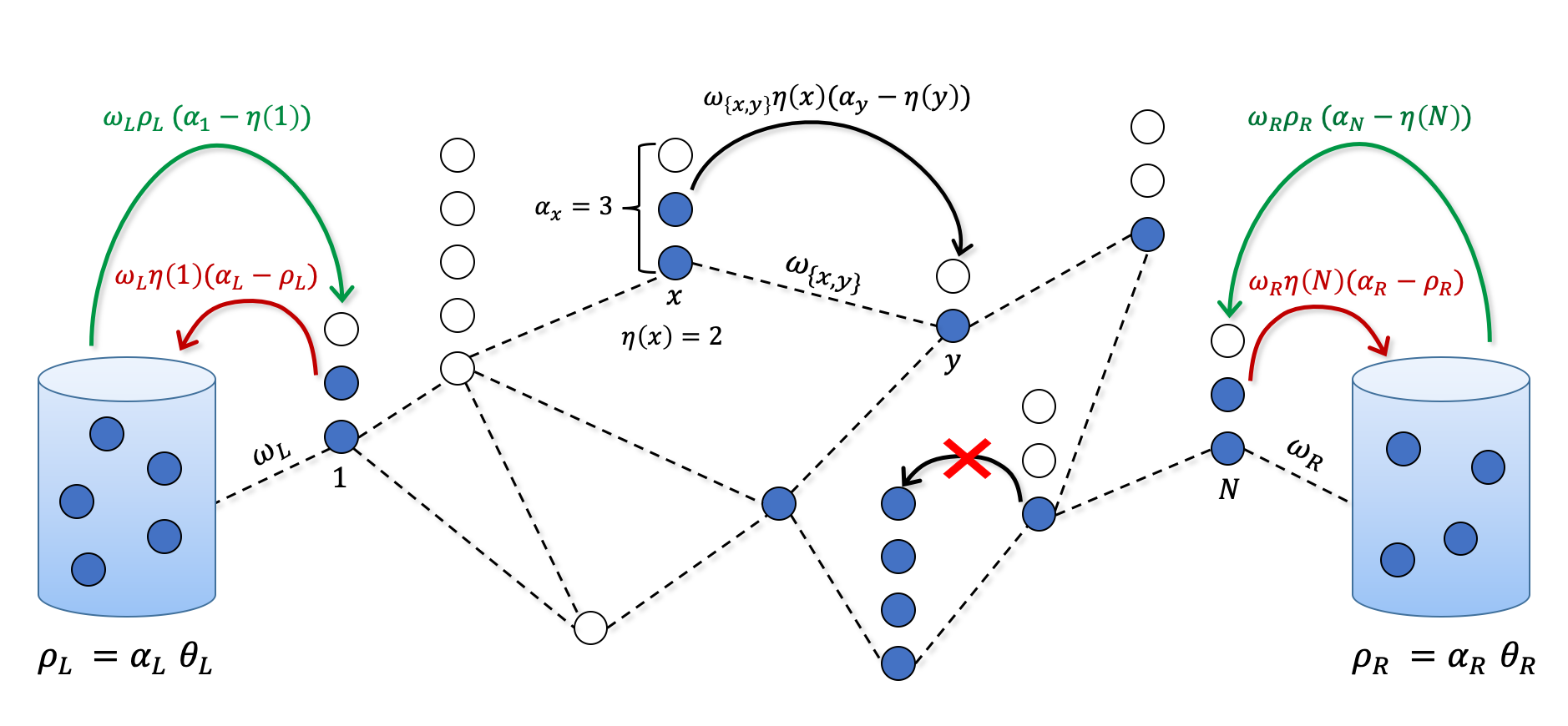}
	\caption{Schematic description of the partial exclusion process ($\SEP$) dynamics in contact with left and right reservoirs.}
	\label{fig:finalimage}
\end{figure}

The parameters $\boldsymbol \alpha = \{\alpha_x: x \in V \} \subset \N$ have the interpretation of \emph{maximal occupancies} for $\SEP$ ($\sigma = -1$) of the sites of $V$ (see Fig.\ \ref{fig:finalimage}). For $\IRW$ ($\sigma = 0$) and $\SIP$ ($\sigma =1$), $\alpha_x \in \N$ stands for the \emph{site attraction parameter} of the site $x \in V$. We observe that the choice  $\boldsymbol \alpha \subset \N$ rather than $(0,\infty)$ is needed only in the context of the exclusion process; however,  for the sake of  uniformity of notation, we adopt  $\N$-valued site parameters $\boldsymbol \alpha$ for all three choices of $\sigma \in \{-1,0,1\}$.

Moreover, while $\cond_L$ and $\cond_R$ shall be interpreted as conductances between the boundaries and the associated bulk sites, the parameters $\alpha_L > 0$ and $\alpha_R > 0$ are the boundary analogues of the bulk site parameters $\boldsymbol \alpha$. The parameters $\scale_L$ and $\scale_R$ are responsible for the scaling of the reservoirs' densities $\rho_L$ and $\rho_R$, i.e.\ 
\begin{align}\label{eq:density_scale_alpha_res}
\rho_L = \alpha_L\scale_L\quad \text{and}\quad \rho_R = \alpha_R \scale_R\ ,
\end{align}
and, for this reason, we refer to them as \emph{scale parameters}. In particular, while in general we only require that
\begin{equation*}
\scale_L\, ,\,\scale_R \in [0,\infty)\ ,
\end{equation*} for the case of the exclusion process ($\sigma = -1$), we need to further impose
\begin{equation*}
\scale_L\, ,\, \scale_R \in [0,1]
\end{equation*}
to prevent the rates in \eqref{eq:L_L} and \eqref{eq:L_R} to become negative.

\begin{remark}[\textsc{notational comparison with \cite{carinci_duality_2013-1}}] If we choose $\cond_{\{x,y\}}=\1_{\{|x-y|=1\}}$ and
	\begin{equation}
	\alpha_x=	\begin{cases}
	2j &\text{if }\sigma=-1\ \text{for some}\  2j\in \N\\
	1& \text{if }\sigma=0\\
	2k&\text{if }\sigma=1\  \text{for some}\  k> 0\ ,
	\end{cases}
	\end{equation}
	for all $x, y \in V$,
	we recover exactly the same bulk dynamics of the models studied in \cite{carinci_duality_2013-1}. For  what concerns the reservoir dynamics, the authors of \cite{carinci_duality_2013-1} employ the following notation (see e.g.\ \cite[Figs.\ 1--2]{carinci_duality_2013-1})
	\begin{align*}
	\begin{split}
	\alpha :=&\ \alpha_L \scale_L\\
	\gamma :=&\ \alpha_R  \scale_R
	\end{split}
	\begin{split}
	\beta :=&\ \alpha_R +\sigma \alpha_R \scale_R\\
	\delta :=&\ \alpha_L + \sigma \alpha_L \scale_L\ .
	\end{split}
	\end{align*}
	However, we believe that the parametrization of the bulk-boundary interaction through $\alpha_L$, $\alpha_R$, $\scale_L$ and $\scale_R$ yields more transparent formulas as, for instance, for the  duality functions in presence of disorder.
\end{remark}

\begin{remark}[\textsc{more general reservoirs geometries}]
	We  emphasize that our results may be stated for boundary driven particle systems with the same bulk dynamics -- as described by the generator $\mathcal L$ -- and a more general boundary part of the dynamics, in which  creation and annihilation of particles due to the reservoir interaction occur at more than two bulk sites. More precisely, the results stated in this section and  Sections \ref{section:measures} and \ref{section:orthogonal_duality} below -- namely, the duality relations and the correlation inequalities -- naturally extend  if $\mathcal L^{L,R}$ in \eqref{eq:generator_full},\eqref{eq:generator_boundary} is replaced by
	\begin{align*}
	\mathcal L^\res f(\eta)= \sum_{x \in V} \cond_x^\res \mathcal L_x^\res f(\eta)\ ,
	\end{align*}
	with
	\begin{align*}
	\mathcal L_x^\res f(\eta):=&\ \eta(x) \left(\alpha^\res_x+\sigma \alpha^\res_x \scale_x^\res\right)\left(f(\eta^{x,-})-f(\eta) \right)\\
	+&\ \alpha^\res_x \scale_x^\res\left(\alpha_x+\sigma \eta(x) \right)\left(f(\eta^{x,+})-f(\eta) \right)\ ,
	\end{align*}
	for some set of non-negative parameters $\boldsymbol \alpha^\res=\{\alpha^\res_x: x\in V\}$, $\boldsymbol \scale^\res=\{\scale^\res_x: x \in V\}$ and $\boldsymbol \cond^\res=\{\cond^\res_x: x \in V\}$. Also the results in Sections \ref{section:n-point} and \ref{section:exponential_generating_functions} below extend to this more general boundary dynamics as long as the scale parameters $\boldsymbol \scale^\res=\{\scale^\res_x: x \in V\}$ take at most two values, say $\scale_L$ and $\scale_R$. 
\end{remark}
\subsection{Duality}
In this section, for each one of the particle systems presented in the section above, we derive two types of duality relations with a particle system in contact with purely absorbing boundaries.  In particular, by duality relation for the particle system $\{\eta_t: t \geq 0\}$ on $\mathcal X$, we mean  that there exist a \emph{dual} particle system $\{\xi_t: t \geq 0 \}$ on $\widehat {\mathcal X}$  and a measurable function $D: \widehat{\mathcal X} \times \mathcal X \to \R$ -- referred to as \emph{duality function} -- for which the following relation holds: for all configurations $\eta \in \mathcal X$, $\xi \in \widehat{\mathcal X}$ and times $t\geq 0$, we have
\begin{align}\label{eq:duality_relation_expectations}
\widehat \E_\xi\left[D(\xi_t,\eta) \right] = \E_\eta\left[D(\xi,\eta_t) \right]\ ,
\end{align}
where $\widehat \E_\xi$, resp.\ $\E_\eta$, denotes expectation w.r.t.\ the law $\widehat \Pr_\xi$ of $\{\xi_t: t \geq 0 \}$ with initial condition $\xi_0=\xi$, resp.\ the law $\Pr_\eta$ of 	$\{\eta_t: t\geq 0\}$ with initial condition $\eta_0=\eta$. More in general, for a given probability measure $\mu$ on $\mathcal X$, $\E_\mu$ denotes the expectation w.r.t.\ the law $\P_\mu$ of $\{\eta_t: t\ge 0 \}$ initially distributed according to $\mu$. Notice that, with a slight abuse of notation, when we  write $\E_\mu\left[D(\xi,\eta) \right]$ we mean $\sum_{\eta \in \mathcal X} D(\xi,\eta)\, \mu(\eta)$.

If $\widehat{\mathcal L}$ and $\mathcal L$ denote the infinitesimal generators associated to the processes $\{\xi_t: t\geq 0\}$ and $\{\eta_t: t \geq 0\}$ respectively, the duality relation \eqref{eq:duality_relation_expectations} is equivalent to the following relation: for all configurations $\eta \in \mathcal X$ and $\xi \in \widehat{\mathcal X}$, we have
\begin{align}\label{eq:duality_relation_generators}
\widehat{\mathcal L}_\eft D(\xi,\eta) = \mathcal L_\ight D(\xi,\eta)\ ,
\end{align}
where the subscript \textquotedblleft$\eft$\textquotedblright, resp.\ \textquotedblleft$\ight$\textquotedblright, indicates that the generator acts as an operator on the function $D(\cdot,\cdot)$, viewed as a function of the left, resp.\ right, variables. More precisely,
\begin{align*}
\widehat{\mathcal L}_\eft D(\xi,\eta) = \widehat{\mathcal L} D(\cdot,\eta)(\xi)\quad \text{and}\quad\mathcal L_\ight D(\xi,\eta)= \mathcal L D(\xi,\cdot)(\eta)\ .	
\end{align*}

In what follows, first we present the dual particle systems and, then, the   duality relations. More specifically, we study in Sections \ref{section:classical_duality} and \ref{section:orthogonal_duality} below, duality relations with  two types of duality functions, which we call, respectively, \textquotedblleft classical\textquotedblright\ and \textquotedblleft orthogonal\textquotedblright\ for reasons that will be explained below.

\subsubsection{Dual particle system with purely absorbing boundaries}\label{section:classical_duality}
For each choice of $\sigma \in \{-1,0,1\}$, we define a particle system  with purely absorbing reservoirs,  which we prove to be dual (see Propositions \ref{proposition:classical_duality} and \ref{proposition:orthogonal_duality} below) to the corresponding  system in contact with reservoirs of Section \ref{section:particle_systems}. For such systems, particles hop on $V \cup \{L,R\}$  following the same bulk dynamics as the particle systems of Section \ref{section:particle_systems} but having $\{L,R\}$ as absorbing sites. More in detail,  $\{\xi_t: t \geq 0\}$ denotes such particle systems having
\begin{equation}
\widehat {\mathcal X}= \mathcal X \times \N_0^{\{L,R\}}
\end{equation} as configuration space and infinitesimal generator $\widehat {\mathcal L}$ given by
\begin{align}\label{eq:generator_full_dual}
\widehat {\mathcal L} f(\xi) = \widehat {\mathcal L}^\bulk f(\xi) + \widehat {\mathcal L}^\bdd f(\xi)\ ,
\end{align}
where, for all bounded functions $f : \widehat {\mathcal X} \to \R$,
\begin{align*}
\widehat {\mathcal L}^\bulk f(\xi) =&\ \sum_{x \sim y} \cond_{\{x,y\}}\,  \widehat {\mathcal L}_{\{x,y\}}f(\xi)\\
=&\ \sum_{x \sim y} \cond_{\{x,y\}} \left\{\begin{array}{c}
\xi(x)\, (\alpha_y+\sigma \xi(y)) \left(f(\xi^{x,y})-f(\xi)\right)\\[.15cm]
+\,\xi(y)\, (\alpha_x+\sigma \xi(x)) \left(f(\xi^{y,x})-f(\xi)\right)
\end{array} \right\}\ ,
\end{align*}
and
\begin{align*}
\widehat{\mathcal L}^\bdd f(\xi) =&\ \cond_L\, \widehat {\mathcal L}_L f(\xi) + \cond_R\, \widehat {\mathcal L}_R f(\xi)\\
=&\ \cond_L\, \alpha_L\, \xi(1) \left(f(\xi^{1,L})-f(\xi)\right)\,  +\, \cond_R\, \alpha_R\, \xi(N) \left(f(\xi^{N,R})-f(\xi)\right) \ ,
\end{align*}
with, for all $x, y \in V \cup \{L,R \}$,  $\xi^{x,y} = \xi-\delta_x+\delta_y  \in \widehat {\mathcal X}$.

For all configurations $\xi \in \widehat{\mathcal X}$, let $|\xi|$ denote the total number of particles of the configuration $\xi$, i.e.\
\begin{align}\label{eq:total_number_of_particles}
|\xi| := \xi(L)+\xi(R) +\sum_{x \in V}  \xi(x)\ .
\end{align}
Once the  total number of particles is fixed, due to the conservation of particles under the dynamics, the  assumption of connectedness of the graph $(V,\sim)$ (see Section \ref{section:particle_systems}) and the positivity of $\cond_L$ and $\cond_R$, the particle system $\{\xi_t: t \geq 0 \}$ is irreducible on
\begin{align*}
\widehat{\mathcal X}_n := \left\{\xi \in \widehat{\mathcal X} : |\xi|= n \right\}
\end{align*} whenever $n = |\xi_0|$ and admits a unique stationary measure fully supported on configurations
\begin{align*}
\left\{\xi \in \widehat{\mathcal X}_n : \xi(x) = 0\ \text{for all}\ x \in V \right\}\ ,
\end{align*}
i.e.\ all particles will get eventually  absorbed  in the sites	 $\{L,R \}$. Furthermore, the evolution of the particle systems $\{\xi_t: t \geq 0\}$ does \emph{not} depend on  $\scale_L$ and $\scale_R$, but only on the following set of parameters:
\begin{equation}\label{eq:environment1}
\boldsymbol \cond = \{\cond_{\{x,y\}}: x, y \in V \}\ ,
\end{equation}
\begin{equation}\label{eq:environment2} \boldsymbol \alpha =\ \{\alpha_x: x \in V \}\qquad \text{and}\qquad \left\{\cond_L, \cond_R, \alpha_L, \alpha_R\right\}\ .
\end{equation}
For this reason, in the sequel we will refer to $V \cup \{L,R\}$ endowed with the above parameters as the \emph{underlying geometry} of our particle systems.

\subsubsection{Classical dualities}
In this section, we generalize to the disordered setting   the duality relations already appearing in e.g.\  \cite{carinci_duality_2013-1}. In particular, these duality functions are in  factorized  -- jointly in the original and dual configuration variables -- form over all sites, i.e., for all $\eta \in \mathcal X$ and $\xi \in \widehat{\mathcal X}$,
\begin{align}\label{eq:duality_factorized}
D(\xi,\eta) = d_L(\xi(L)) \times \left(\prod_{x \in V} d_x(\xi(x),\eta(x))\right) \times d_R(\xi(R))\ ,
\end{align}
with the factors $\{d_x(\cdot,\cdot): x \in V\} \cup \{d_L(\cdot), d_R(\cdot) \}$  named \emph{single-site duality functions}. Moreover, we refer to them as \textquotedblleft classical\textquotedblright\ because the duality functions  consist in   weighted factorial moments of the occupation variables of the configuration $\eta$ generalizing to $\IRW$ and $\SIP$  the 	 renown duality relations for the symmetric simple exclusion process, see e.g.\ \cite[Theorem 1.1, p.\ 363]{liggett_interacting_2005-1}.	

The precise form of these classical duality functions is contained in the following proposition. The proof of this duality relation boils down to  directly check identity \eqref{eq:duality_relation_generators} and we omit it  being it a straightforward rewriting  of the proof of \cite[Theorem 4.1]{carinci_duality_2013-1}. We remark that in \eqref{eq:single_site_classical_LR} below and in the rest of the paper, we adopt the conventions $0^0:=1$,  $\frac{\Gamma(v+\ell) }{\Gamma(v)}:=v(v+1)\cdots(v+\ell-1)$ for $v\ge 0$ and $\ell\in \N_0$, while $$\frac{\Gamma(v+\ell) }{\Gamma(v)}:=\begin{cases}1\ &\text{if }\ell=0\\
v(v+1)\cdots(v+\ell-1)\ & \text{if }\ell\in\{1,...,\abs{v}\} \\
0\ &\text{otherwise ,}
\end{cases}$$
for $v\in\Z\cap(-\infty,0)$ and $\ell\in \N_0$. 
\begin{proposition}[\textsc{classical duality functions}]\label{proposition:classical_duality}	For each choice of $\sigma \in \{-1,0,1\}$, let $\mathcal L$ and $\widehat{\mathcal L}$ be the infinitesimal generators given in \eqref{eq:generator_full} and \eqref{eq:generator_full_dual}, respectively, associated to the particle systems $\{\eta_t: t \geq 0 \}$ and $\{\xi_t: t \geq 0\}$. Then the duality relations in \eqref{eq:duality_relation_expectations} and \eqref{eq:duality_relation_generators} hold with the   duality function $D^\cl : \widehat{\mathcal X} \times \mathcal X \to \R$ defined as follows:  for all configurations $\eta \in \mathcal X$ and $\xi \in \widehat{\mathcal X}$,
	\begin{align*}
	D^\cl(\xi,\eta) = d^\cl_L(\xi(L)) \times \left(\prod_{x \in V} d^\cl_x(\xi(x),\eta(x))\right) \times d^\cl_R(\xi(R))\ ,
	\end{align*} where, for all $x \in V$ and $k, n \in \N_0$,
	\begin{align}\label{eq:single-site_classical}
	d^\cl_x(k,n) = \frac{n!}{(n-k)!}\frac{1}{w_x(k)}\, \1_{\{k \leq n\}}
	\end{align}
	and
	\begin{align}\label{eq:single_site_classical_LR}
	d^\cl_L(k) = (\scale_L)^k\qquad \text{and}\qquad d^\cl_R(k) = (\scale_R)^k\ ,
	\end{align}
	where
	\begin{align}\label{eq:wx}
	w_x(k)\ =\ \begin{dcases}
	\frac{\alpha_x!}{(\alpha_x-k)!}\1_{\{k \leq \alpha_x\}}&\text{if\ } \sigma = -1\\
	\alpha_x^k &\text{if } \sigma = 0\\
	\frac{\Gamma(\alpha_x+k)}{\Gamma(\alpha_x)} &\text{if } \sigma = 1\ .
	\end{dcases}
	\end{align}
\end{proposition}
\section{Equilibrium and non-equilibrium stationary measures}\label{section:measures}

The long run behavior of  the boundary driven particle systems of Section \ref{section:particle_systems}, encoded in their stationary measures, is explicitly known when the particle systems are \emph{not} in contact with the reservoirs.
Indeed, if $\cond_L = \cond_R = 0$, the particle systems $\{\eta_t: t\geq 0 \}$ admit a one-parameter family of stationary -- actually reversible -- product measures
\begin{equation}\label{eq:muscale}
\{\mu_\scale = \otimes_{x \in V}\, \nu_{x,\scale} : \scale \in \Scale \}
\end{equation} with
$\Scale = [0,1]$ if $\sigma = -1$ ($\SEP$) and $\Scale = [0,\infty)$ if $\sigma = 0$ ($\IRW$) and $\sigma = 1$ ($\SIP$) and marginals given, for all $x \in V$, by
\begin{align}\label{eq:marginals}
\nu_{x,\scale} \sim \begin{dcases}  \Bin(\alpha_x,\scale) &\text{if } \sigma = -1\\[.15cm]
\Poi(\alpha_x \scale) &\text{if } \sigma = 0\\[.15cm]
\NegBin(\alpha_x, \tfrac{\scale}{1+\scale}) &\text{if } \sigma = 1\ .
\end{dcases}
\end{align}
More concretely, for all $n \in \N_0$,
\begin{align}\label{eq:explicitnu}
\nu_{x,\scale}(n)\ =\ \frac{w_x(n)}{z_{x,\scale}}\, \frac{\left(\frac{\scale}{1+\sigma \scale} \right)^n}{n!}\ ,
\end{align}
with the functions $\{w_x: x \in V\}$ as given in \eqref{eq:wx} and
\begin{align}\label{eq:z}
z_{x,\scale}\ =\ \begin{dcases}
(1-\scale)^{-\alpha_x}&\text{if } \sigma = -1\\
e^{\alpha_x \scale} &\text{if } \sigma = 0\\
(1+\scale)^{\alpha_x} &\text{if } \sigma = 1\ ,
\end{dcases}
\end{align}
where, for $\sigma =-1$, we set $\nu_{x,1}(n):=\1_{\{n=\alpha_x\}}$.
Reversibility of these product measures for the dynamics induced by $\mathcal L$ in \eqref{eq:generator_full} follows by a standard detailed balance computation  (see e.g.\  \cite{carinci_duality_2013-1} for an analogous statement with site-independent parameters $\boldsymbol \alpha$). We note that, in analogy with \eqref{eq:density_scale_alpha_res}, the parameterization of these product measures and corresponding marginals is chosen in such a way that the density of particles
\begin{equation}
\rho_x:= \E_{\mu_\scale}[\eta(x)]
\end{equation}
at site $x \in V$ w.r.t.\ $\mu_\scale$ is given by the product of $\alpha_x$ and $\scale$, i.e.
\begin{equation}
\rho_x= \alpha_x \scale\ ,\quad x \in V\ .
\end{equation}

\subsection{Equilibrium stationary measure}
In presence of interaction with only one of the two reservoirs, e.g.\ $\cond_L > 0$ and $\cond_R = 0$ and with scale parameters given by $\scale_L$ and $\scale_R$, respectively, the same detailed balance computation shows that the systems have  $\mu_\scale$ (see \eqref{eq:muscale})  with $\scale = \scale_L$  as reversible product measures. 	The stationary measures remain the same as long as the systems are in contact with both reservoirs, i.e.\ $\cond_L, \cond_R > 0$, and the two reservoirs are given equal scale parameters $\scale_L= \scale_R \in \Scale$. We refer to such stationary measures as \emph{equilibrium stationary measures}.

\subsection{Non-equilibrium stationary measures}\label{section:non-equilibrium}
As for \emph{non-equilibrium stationary measures}, i.e.\ the  stationary measures of the particle systems when $\cond_L, \cond_R > 0$ and $\scale_L \neq \scale_R$, none of the measures $\{\mu_\scale = \otimes_{x \in V}\, \nu_{x,\scale} : \scale \in \Scale \}$ above is stationary. However,  for each of the particle systems,  the  non-equilibrium stationary measure
exists, is unique and we denote it by $\mu_{\scale_L,\scale_R}.$ Moreover, while for the case of independent random walkers $\mu_{\scale_L,\scale_R}$ is in product form, for the case of exclusion and inclusion particle systems in non-equilibrium $\mu_{\scale_L,\scale_R}$ is non-product and has non-zero two-point correlations. This is the content of Theorem \ref{theorem:existence_uniqueness} below. In particular, the result on two-point correlations (item \ref{it:sepsip}) will be complemented with the study of the signs of such correlations in Theorem \ref{theorem:sign_two_point} and Lemma \ref{lemma:sign_two_point} below. We recall that, for the special case of the exclusion process with $\boldsymbol \alpha = \{\alpha_x : x \in V \}$ satisfying $\alpha_x = 1$ for all $x \in V$ and with nearest-neighbor unitary conductances, i.e.\
\begin{align*}
\cond_{\{x,y\}}= \1_{\{|x-y|=1\}}\ ,\quad x, y \in V\ ,
\end{align*}
the unique non-product non-equilibrium stationary measure $\mu_{\scale_L,\scale_R}$ has been characterized in terms of a matrix formulation (see e.g.\  \cite{derrida_exact_1993-1} and \cite[Part III. Section 3]{liggett_stochastic_1999}). Goal of Section \ref{section:n-point} below is to provide a partial characterization of the non-equilibrium stationary measure of these systems by expressing suitably centered factorial moments -- related to the orthogonal duality functions of Section \ref{section:orthogonal_duality} below -- in terms of the product of a suitable power of $\left(\scale_L-\scale_R \right)$ and a coefficient which does not depend on neither $\scale_L$ nor $\scale_R$.

In what follows, for all $x \in V$, we introduce the non-equilibrium stationary profile of the classical duality functions:
\begin{align}\label{eq:scalex}
\bar \scale_x := \E_{\mu_{\scale_L,\scale_R}}\left[\frac{\eta(x)}{\alpha_x} \right] =\E_{\mu_{\scale_L,\scale_R}}\left[D^\cl(\delta_x,\eta) \right]\	  .
\end{align}
We recall that $\widehat \Pr_\xi$ denotes the law of the dual particle system started from the deterministic configuration $\xi \in \widehat{\mathcal X}$. Then, by stationarity and duality (Proposition \ref{proposition:classical_duality}), we obtain, for all $x \in V$,
\begin{equation}\label{eq:scalexxx}
\bar \scale_x =  \lim_{t\to \infty} \E_{\mu_{\scale_L,\scale_R}}\left[\frac{\eta_t(x)}{\alpha_x} \right] =  \widehat p_\infty(\delta_x,\delta_L)\,\scale_L + \widehat p_\infty(\delta_x,\delta_R)\, \scale_R 
= \scale_R + \widehat p_\infty(\delta_x,\delta_L)\, (\scale_L-\scale_R)\ ,
\end{equation}
where, for all $\xi, \xi' \in \widehat {\mathcal X}$, $\widehat p_\infty(\xi,\xi'):= \lim_{t \to \infty} \widehat p_t(\xi,\xi')$, with
\begin{align*}
\widehat p_t(\xi,\xi') := \widehat \Pr_\xi(\xi_t = \xi')\ .
\end{align*}
Equivalently, stationarity and duality imply that $\{\bar \scale_x: x \in V \}$ solves the following difference equations: for all $x \in V$,
\begin{align}\label{eq:harmonic_theta}
  \sum_{y \in V} \cond_{\{x,y\}}\, \alpha_y\, (\bar\scale_y-\bar\scale_x) +  \1_{\{x=1\}}\, \cond_L\, \alpha_L\, (\scale_L-\bar\scale_1) + \1_{\{x=N\}}\, \cond_R\, \alpha_R\, (\scale_R-\bar\scale_N) =  0\ .
\end{align}
Consequently, because of the connectedness of $(V,\sim)$,   if $\scale_L=\scale_R$, then $\bar \scale_x = \scale_L=\scale_R$ for all $x \in V$, while  $\scale_L\neq \scale_R$ implies that there  exist $x, y \in V$ such that $\bar \scale_x \neq \bar \scale_y$ and, moreover, that $\bar \scale_x > 0$ for all $x \in V$.

\begin{remark}[\textsc{non-equilibrium stationary profile for a chain}]
	In the particular instance of a chain, i.e.
	\begin{align*}
	\cond_{\{x,y\}}>0\quad \text{if and only if}\quad |x-y|=1\ ,
	\end{align*}the solution to the system \eqref{eq:harmonic_theta} is given by:
	\begin{align*}
	\bar \scale_x&= \scale_R + \widehat p_\infty(\delta_x,\delta_L)\left(\scale_L-\scale_R \right)\\
	&= \scale_R + \left( \frac{    \frac{1}{\cond_R \alpha_R \alpha_N}  + \sum_{y=x}^{N-1} \frac{1}{\cond_{\{y,y+1\}}\alpha_y \alpha_{y+1}} }{\frac{1}{\cond_L \alpha_L \alpha_1} + \left(\sum_{y=1}^{N-1} \frac{1}{\cond_{\{y,y+1\}}\alpha_y \alpha_{y+1}}\right) + \frac{1}{\cond_R \alpha_R \alpha_N}}  \right)\left(\scale_L-\scale_R \right)\ .
	\end{align*}
	If, additionally, the conductances and site parameters $\boldsymbol \cond$ and $\boldsymbol \alpha$ are constant, $\alpha_L= \alpha_R= \alpha_x$ and $\cond_L = \cond_R = \cond_{\{x,x+1	\}}$, the profile $x\mapsto \bar \scale_x$ is linear (cf.\  \cite[Eq.\ (4.24)]{carinci_duality_2013-1}):
	\begin{equation}\label{eq:linear_profile}
	\bar \scale_x= \scale_R + \left(1-\frac{x}{N+1}\right) \left(\scale_L-\scale_R \right)\ .
	\end{equation}
\end{remark}

Before stating the main result of this section, we introduce the following  definition.
\begin{definition}[\textsc{local equilibrium product measure}]\label{def: local equilibrium product measure}
	Given $\bar{\boldsymbol \scale}:=\{ \bar \scale_x:x \in V\}$  the stationary profile introduced in \eqref{eq:scalexxx}, we define the following product  measure
	\begin{align}\label{eq:muLR2}
	\mu_{\bar{\boldsymbol \scale}} := \otimes_{x \in V}\, \nu_{x,\bar \scale_x}\ ,
	\end{align}
	and  refer to it as the \emph{local equilibrium product measure}.
\end{definition}

\begin{theorem}\label{theorem:existence_uniqueness}
	For each choice of $\sigma \in \{-1,0,1\}$ and  provided that  $\cond_L\vee\cond_R > 0$,	 for all $\scale_L, \scale_R \in \Scale$ there exists a unique  stationary measure $\mu_{\scale_L,\scale_R}$ for the particle system $\{\eta_t: t \geq 0\}$.  Moreover,
	\begin{enumerate}[label={\normalfont(\alph*)},ref={\normalfont(\alph*)}]
		\item \label{it:irw}If $\sigma = 0$ $(\IRW)$, the stationary measure $\mu_{\scale_L,\scale_R}$ is in product form	 and is given by
		\begin{align}\label{eq:mustat=poisson}
		\mu_{\scale_L,\scale_R} = \mu_{\bar{\boldsymbol \scale}}\ .
		\end{align}

		\item \label{it:sepsip}If either $\sigma = -1$ $(\SEP)$ or $\sigma = 1$ $(\SIP)$ and, additionally, $\cond_L, \cond_R > 0$ and $\scale_L \neq \scale_R$, there exists $x, y \in V$ with $x \neq y$ for which
		\begin{align*}
		\E_{\mu_{\scale_L,\scale_R}}\left[\left(\frac{\eta(x)}{\alpha_x}-\bar \scale_x \right)\left(\frac{\eta(y)}{\alpha_y}-\bar \scale_y  \right) \right] \neq 0\ .
		\end{align*}
		As a consequence,   the unique non-equilibrium stationary measure $\mu_{\scale_L,\scale_R}$ is not in product form.
	\end{enumerate}
\end{theorem}
\begin{proof}
	The proof of existence and uniqueness of the stationary measure $\mu_{\scale_L,\scale_R}$ is trivial for the exclusion process, which is a finite state irreducible Markov chain. We postpone the proof for the case of independent random walkers and inclusion process to Appendix \ref{appendix:existence_uniqueness}. Although  this result is  standard, it does not appear, to the best of our knowledge, in the literature.

	For what concerns item \ref{it:irw} in which $\sigma =0$, let us compute, for all $\xi \in \widehat{\mathcal X}$,
	\begin{align*}
	\sum_{\eta \in \mathcal X} \mathcal L_\ight D^\cl(\xi,\eta)\, \mu_{\bar{\boldsymbol \scale}}(\eta)\ .
	\end{align*}
	By duality, the  following
	relation (cf.\ e.g.\ \cite{redig_factorized_2018})
	\begin{align}\label{eq:fundamental_relation}
	\sum_{n \in \N_0} d^\cl_x(k,n)\, \nu_{x,\bar \scale_x}(n) = (\bar \scale_x)^k\ , 
	\end{align}
which holds for all $x \in V$ and $k \in \N_0$ if $\sigma \in \{0,1\}$ while $k \in \{0,\ldots, \alpha_x\}$ if $\sigma = -1$,
  we obtain, for all $\xi \in \widehat{\mathcal X}$, 	
	\begin{align*}
	&	\sum_{\eta \in \mathcal X} \mathcal L_\ight D^\cl(\xi,\eta)\, \mu_{\bar{\boldsymbol \scale}}(\eta) = 	\sum_{\eta \in \mathcal X} \widehat {\mathcal L}_\eft D^\cl(\xi,\eta)\, \mu_{\bar{\boldsymbol \scale}}(\eta)\\
	&= \sum_{x \in V} \left(	\sum_{\eta \in \mathcal X} D^\cl(\xi-\delta_x,\eta)\, \mu_{\bar{\boldsymbol \scale}}(\eta)\right)  \xi(x)
	\left\{
	\begin{array}{l} \sum_{y \in V} \cond_{\{x,y\}}\, \alpha_y\, (\bar \scale_y-\bar \scale_x)\\[.15cm]
	+\1_{\{x=1\}}\, \cond_L\, \alpha_L\, (\scale_L-\bar \scale_1)\\[.15cm]
	+\1_{\{x=N\}}\, \cond_R\, \alpha_R\, (\scale_R-\bar\scale_N)
	\end{array}
	\right\}=0\ ,
	\end{align*}
where  the last identity follows from 	 \eqref{eq:harmonic_theta}.
	Because the products of Poisson distributions are completely characterized by their factorial moments $\{D^\cl(\xi,\cdot): \xi \in \widehat{\mathcal X} \}$, we get \eqref{eq:mustat=poisson}.
	
	\
	
	For item \ref{it:sepsip} in which $\sigma \neq 0$,  let us suppose by contradiction that all two-point correlations are zero, i.e.\   for all $x, y \in V$ with $x \neq y$,
	\begin{align}\label{eq:two_point_zero}
	\E_{\mu_{\scale_L,\scale_R}}\left[\frac{\eta(x)}{\alpha_x} \frac{\eta(y)}{\alpha_y} \right]= \E_{\mu_{\scale_L,\scale_R}}\left[D^\cl(\delta_x+\delta_y,\eta) \right]=\bar \scale_x \bar \scale_y\ .
	\end{align}
	If we use the following shortcut
	\begin{align*}
	\bar \scale_x'':= \E_{\mu_{\scale_L,\scale_R}}\left[D^\cl(2\delta_x,\eta) \right]\ ,
	\end{align*}
	by stationarity, duality and \eqref{eq:two_point_zero}, we obtain, for all $x \in V$,
	\begin{align*}
	&	\sum_{\eta \in \mathcal X} \mathcal L_\ight D^\cl(2\delta_x,\eta)\, \mu_{\scale_L,\scale_R}(\eta) = 	\sum_{\eta \in \mathcal X} \widehat{\mathcal L}_\eft D^\cl(2\delta_x,\eta)\, \mu_{\scale_L,\scale_R}(\eta)\\
	&=2 \sum_{y \in V} \cond_{\{x,y\}} \alpha_y (\bar \scale_x\, \bar\scale_y- \bar\scale_x'' )+ 2\left\{\1_{\{x=1\}}\cond_L\alpha_L (\scale_L \bar\scale_1-\bar \scale_1'' ) + \1_{\{x=N\}}\cond_R\alpha_R (\scale_R \bar\scale_N-\bar \scale_N'' )\right\} = 0\ .
	\end{align*}
	By adding and subtracting
	\begin{align*}
	&2 \left\{\sum_{y \in V} \cond_{\{x,y\}} \alpha_y (\bar \scale_x)^2+ \1_{\{x=1\}}\cond_L\alpha_L (\bar \scale_1)^2  + \1_{\{x=N\}}\cond_R\alpha_R(\bar \scale_N)^2 \right\}
	\end{align*}
	to the identity above and by relation \eqref{eq:harmonic_theta}, we get
	\begin{align*}
	\left((\bar \scale_x)^2-\bar \scale_x''\right) 2\left\{\sum_{y \in V} \cond_{\{x,y\}}\alpha_y+\1_{\{x=1\}}\cond_L\alpha_L +\1_{\{x=N\}}\cond_R \alpha_R \right\} = 0\ .
	\end{align*}
	Because the above identity holds for all $x \in V$ and by the positivity of the expression in curly brackets due to the connectedness of $(V,\sim)$, we get
	\begin{align}\label{eq:two_point_zero2}
	\bar \scale_x''= (\bar \scale_x)^2\ ,\quad \text{for all}\quad x\in V\ .
	\end{align}
	In view of \eqref{eq:two_point_zero}, \eqref{eq:two_point_zero2}, stationarity of $\mu_{\scale_L,\scale_R}$ and duality, we have
	\begin{align}\label{eq:second_moment}
	\nonumber
	&\sum_{\eta \in \mathcal X}\mathcal L_\ight D^\cl(\delta_x+\delta_y,\eta)\, \mu_{\scale_L,\scale_R}(\eta)= 	\sum_{\eta \in \mathcal X} \widehat {\mathcal L}_\eft D^\cl(\delta_x+\delta_y,\eta)\, \mu_{\scale_L,\scale_R}(\eta)\\
	\nonumber
	&=	\bar \scale_y
	\left\{\begin{array}{l}\sum_{z \in V} \cond_{\{x,z\}}\, \alpha_z\, (\bar \scale_z-\bar \scale_x)\\[.15cm]
	+\1_{\{x=1\}}\, \cond_L\, \alpha_L\, (\scale_L-\bar \scale_1)\\[.15cm]
	+\1_{\{x=N\}}\, \cond_R\, \alpha_R\, (\scale_R-\bar \scale_N)\end{array} \right\}
	+ \bar \scale_x \left\{\begin{array}{l}\sum_{z \in V} \cond_{\{y,z\}}\, \alpha_z\, (\bar \scale_z-\bar \scale_y)\\[.15cm]
	+\1_{\{y=1\}}\, \cond_L\, \alpha_L\, (\scale_L-\bar \scale_1)\\[.15cm]
	+\1_{\{y=N\}}\, \cond_R\, \alpha_R\, (\scale_R-\bar \scale_N)\end{array} \right\}
	+  \sigma \cond_{\{x,y\}} (\bar \scale_x-\bar \scale_y)^2\\
	&=	 \sigma \cond_{\{x,y\}} (\bar \scale_x-\bar \scale_y)^2\ .
	\end{align}
	Therefore, because $\sigma \in \{-1,1\}$, as a consequence of  the connectedness  of $(V,\sim)$, we have	
	\begin{align}\label{eq:zero2}
	&\sum_{x \sim y} \left(	\sum_{\eta \in \mathcal X} \mathcal L_\ight D^\cl(\delta_x+\delta_y,\eta)\, \mu_{\scale_L,\scale_R}(\eta) \right)	 = \sigma \sum_{x \sim y} \cond_{\{x,y\}}\,  (\bar \scale_x-\bar \scale_y)^2 = 0
	\end{align}
	if and only if
	\begin{align}\label{eq:zero1}
	\bar \scale_x = \bar \scale_y\ ,\quad \text{for all}\quad x, y \in V\ .
	\end{align}
	However, because $\scale_L \neq \scale_R$, the latter condition \eqref{eq:zero1} contradicts the claim below \eqref{eq:harmonic_theta} .

\end{proof}

\subsection{Two-point correlations in the non-equilibrium steady state}
In the following theorem we prove that as soon as the system has interaction, i.e.\
$\sigma\in \{-1, 1\}$, the local equilibrium product measure expectations of classical duality functions decrease (resp.\ increase) for exclusion (resp.\  inclusion) in the course of time. This implies,  in particular, negative (resp.\ positive) two-point correlations for exclusion (resp.\ inclusion) particle systems. This strengthens previous results on correlation inequalities in \cite{giardina_correlation_2010}, indeed here we obtain strict inequalities. The proof of this theorem is based on Lemma \ref{lemma:sign_two_point} below, which is of interest in itself because it provides an explicit expression of the l.h.s.\ in \eqref{eq:derivative_expectation}.

\begin{theorem}[\textsc{sign of two-point correlations}]\label{theorem:sign_two_point}
	If $\cond_L,  \cond_R > 0$  and  $\xi\in \widehat{\mathcal X}$ is  such that $\sum_{x\in V}\xi(x)\ge 2$, then, for all $\scale_L, \scale_R \in \Scale$ with $\scale_L \neq \scale_R$
	and $t>0$,
	\begin{align}\label{eq:derivative_expectation}
	&	\frac{\dd}{\dd t}	\E_{\mu_{\bar{\boldsymbol \scale}}}\left[D^\cl(\xi,\eta_t) \right]\
	\begin{dcases}
	<0 &\text{if } \sigma = -1\\
	> 0 &\text{if } \sigma = 1\ .
	\end{dcases}
	\end{align}
	As a consequence, for all $x, y \in V$ with $x \neq y$, 	
	\begin{align*}
	\E_{\mu_{\scale_L,\scale_R}}\left[\left(\frac{\eta(x)}{\alpha_{x}} - \bar \scale_{x}\right) \left(\frac{\eta(y)}{\alpha_{y}} - \bar \scale_{y}\right)\right]\  \begin{dcases}
	<0 &\text{if } \sigma = -1\\
	> 0 &\text{if } \sigma = 1\ .
	\end{dcases}
	\end{align*}
\end{theorem}
\begin{proof}
	The local equilibrium product measures $\mu_{\bar{\boldsymbol \scale}}$ satisfy the hypothesis of  Lemma  \ref{lemma:sign_two_point} below (cf.\ \eqref{eq:fundamental_relation}). Then, by the claim after \eqref{eq:harmonic_theta} and the assumption $\scale_L \neq \scale_R$, 	 \eqref{eq:derivative_expectation} is recovered as a consequence of the first equality of  \eqref{eq:time_derivative} from the same lemma.
\end{proof}

\begin{lemma}\label{lemma:sign_two_point}\label{it:study_sign} For all $n \in \N$, let $\mu$ be a probability measure on $\mathcal X$ such that
	\begin{align}\label{eq:assumption_n_order}
	\E_{\mu}\left[D^\cl(\xi,\eta) \right] &= H(\xi,\bar {\boldsymbol \scale})
	\end{align}
	holds for all $\xi \in \widehat{\mathcal X}$ with $|\xi|\leq n$, where $\bar{\boldsymbol \scale} = \{\bar \scale_x: x \in V \}$ and, for all $\boldsymbol \scale = \{ \scale_x : x \in V\} \subset \Scale$,
	\begin{align*}
	H(\xi,\boldsymbol \scale) := (\scale_L)^{\xi(L)} \left( \prod_{x\in V} (\scale_x)^{\xi(x)}\right) (\scale_R)^{\xi(R)}\ .
	\end{align*}
	Then
	\begin{align}\label{eq:time_derivative}
	\nonumber
	\frac{\dd}{\dd t}\E_\mu\left[D^\cl(\xi,\eta_t) \right]
	&=\sigma \sum_{x \sim y} \cond_{\{x,y\}}\left( \bar \scale_y-\bar \scale_x\right)^2 \widehat \E_\xi\left[\frac{\xi_t(x)}{\bar \scale_x} \frac{\xi_t(y)}{\bar \scale_y}  \E_{\mu}\left[D^\cl(\xi_t,\eta) \right]\right]\\
	&=\sigma \sum_{x \sim y} \cond_{\{x,y\}}\, \widehat \E_\xi\left[\left( \bar \scale_y-\bar \scale_x\right)^2\partial^2_{\scale_x\scale_y} H(\xi_t,\bar{\boldsymbol \scale}) \right]
	\end{align}
	holds for all $\xi \in \mathcal X$ with $|\xi|\leq n$ and $t \geq 0$.
\end{lemma}

\begin{proof}
	By duality, we obtain, for all $\xi \in \widehat{\mathcal X}$,
	\begin{align*}
	&\frac{\dd}{\dd t}\E_\mu\left[D^\cl(\xi,\eta_t) \right] = 	\sum_{\eta \in \mathcal X} \mathcal L_\ight \E_\eta\left[D^\cl(\xi,\eta_t)\right] \mu(\eta)\\
	&= 	\sum_{\eta \in \mathcal X} \mathcal L_\ight \widehat \E_\xi\left[D^\cl(\xi_t,\eta) \right] \mu(\eta) = 	\sum_{\eta \in \mathcal X} \widehat \E_\xi\left[ \widehat{\mathcal L}_\eft D^\cl(\xi_t,\eta)\right] \mu(\eta)\\
	&= \sum_{x \in V} \widehat \E_\xi\left[\xi_t(x) \left\{\begin{array}{c}
	\sum_{y \in V} \cond_{\{x,y\}}\, \alpha_y
	\left(\E_\mu\left[D^\cl(\xi_t^{x,y},\eta)\right]-\E_\mu\left[D^\cl(\xi_t,\eta) \right]  \right)\\[.15cm]
	+\, \1_{\{x=1\}}\cond_L\, \alpha_L\, (\E_\mu\left[D^\cl(\xi_t^{1,L},\eta)  \right] -\E_\mu\left[D^\cl(\xi_t,\eta)  \right] )\\[.15cm]
	+\, \1_{\{x=N\}}\cond_R\, \alpha_R\, (\E_\mu\left[D^\cl(\xi_t^{N,R},\eta)  \right] -\E_\mu\left[D^\cl(\xi_t,\eta)  \right] )
	\end{array} \right\} \right]\\
	&\, +\sigma  \sum_{x \in V} \widehat \E_\xi\left[\sum_{y \in V}\cond_{\{x,y\}} \xi_t(x)\,   \xi_t(y)\left(\E_\mu\left[D^\cl(\xi_t^{x,y},\eta)\right]-\E_\mu\left[D^\cl(\xi_t,\eta) \right]  \right) \right]\ .
	\end{align*}
	By \eqref{eq:assumption_n_order}, for all $x, y \in V$ and $\xi \in \widehat {\mathcal X}$ with $|\xi|\leq n$, we have
	\begin{align*}
	\E_\mu\left[D^\cl(\xi^{x,y},\eta)\right]-\E_\mu\left[D^\cl(\xi,\eta) \right]   = \frac{\E_\mu\left[D^\cl(\xi,\eta) \right]}{\bar \scale_x} \left(\bar \scale_y - \bar \scale_x \right)\ ,
	\end{align*}
	and, similarly,
	\begin{align*}
	\E_\mu\left[D^\cl(\xi^{1,L},\eta) \right]-\E_\mu\left[D^\cl(\xi,\eta)\right]&= \frac{\E_\mu\left[D^\cl(\xi,\eta) \right]}{\bar \scale_1}\left(\scale_L-\bar \scale_1 \right)\\
	\E_\mu\left[D^\cl(\xi^{N,R},\eta) \right]-\E_\mu\left[D^\cl(\xi,\eta)\right]&= \frac{\E_\mu\left[D^\cl(\xi,\eta) \right]}{\bar \scale_N}\left(\scale_R-\bar \scale_N \right)\ .	
	\end{align*}
	As a consequence, we further obtain
	\begin{align*}
	&\frac{\dd}{\dd t}\E_\mu\left[D^\cl(\xi,\eta_t) \right] = 	\sum_{\eta \in \mathcal X} \mathcal L_\ight \E_\eta\left[D^\cl(\xi,\eta_t)\right] \mu(\eta)\\
	&=\sum_{x \in V} \widehat \E_\xi\left[\frac{\xi_t(x)}{\bar \scale_x}\, \E_\mu\left[D^\cl(\xi_t,\eta) \right] \left\{\begin{array}{l}
	\sum_{y \in V}\cond_{\{x,y\}}\, \alpha_y \left(\bar \scale_y-\bar \scale_x \right)\\[.15cm]
	+ \1_{\{x=1\}}\cond_L \alpha_L\left(\scale_L-\bar \scale_1 \right)\\[.15cm]
	+\1_{\{x=N\}}\cond_R \alpha_R \left(\scale_R-\bar \scale_N \right)
	\end{array} \right\} \right]\\
	&\,+\sigma \sum_{x \sim y} \cond_{\{x,y\}}\left( \bar \scale_y-\bar \scale_x\right)^2 \widehat \E_\xi\left[\frac{\xi_t(x)}{\bar \scale_x} \frac{\xi_t(y)}{\bar \scale_y}  \E_{\mu}\left[D^\cl(\xi_t,\eta) \right]\right]\ .
	\end{align*}
	The observation that each of the expressions between curly brackets above equals zero because of the choice of the scale parameters $\{\bar \scale_x :  x \in V \}$ (cf.\ \eqref{eq:scalex} and \eqref{eq:harmonic_theta}) concludes the proof.\end{proof}

\begin{remark} \begin{enumerate}[label={\normalfont(\alph*)}]
		\item For all $\xi \in \widehat{\mathcal X}$ with $\sum_{z \in V} \xi(z) \geq 2$, for all times $t >0$ and for all sites $x, y \in V$, the geometric assumption on the connectedness of $(V,\sim)$ implies that
		\begin{equation*}
		\widehat \Pr_\xi\left(\xi_t(x) \xi_t(y)>0 \right)>0\ .
		\end{equation*}
		As a consequence,   the sign of the time derivative in \eqref{eq:time_derivative} for $\xi \in \widehat {\mathcal X}$ with $\sum_{z \in V} \xi(z) \geq 2$ and for $t > 0$ is determined by $\sigma \in \{-1,0,1\}$. In particular, if the probability measure $\mu$ and the configuration $\xi \in \widehat {\mathcal X}$ are given as in Theorem \ref{theorem:sign_two_point}, the convergence
		\begin{align*}
		\E_\mu\left[D^\cl(\xi,\eta_t) \right] \underset{t\to \infty}\longrightarrow \E_{\mu_{\scale_L,\scale_R}}\left[D^\cl(\xi,\eta) \right]
		\end{align*}
		is \emph{strictly monotone} in time: decreasing for $\sigma = -1$ and increasing for $\sigma =1$.
		
		\item In the particular situation in which $\xi = \delta_x+\delta_y$ for some $x, y \in V$ and the probability measure $\mu$ satisfies the hypothesis of Theorem \ref{theorem:sign_two_point} for $n \geq 2$, the expression in \eqref{eq:time_derivative} further simplifies yielding, for all $t > 0$,
		\begin{align}\label{eq:two_particle_duality_formula}\nonumber
		&\E_\mu\left[D^\cl(\delta_x+\delta_y,\eta_t) \right]- \bar \scale_x \bar \scale_y\\
		\nonumber
		&=\E_\mu\left[D^\cl(\delta_x+\delta_y,\eta_t) \right] - \E_\mu\left[D^\cl(\delta_x+\delta_y,\eta_0) \right]\\
		\nonumber
		&= \sigma \int_0^t \sum_{z \sim w} \cond_{\{z,w\}} \left(\bar \scale_w-\bar \scale_z \right)^2 \widehat \E_{\xi=\delta_x+\delta_y}\left[\xi_s(z)\xi_s(w) \right]\dd s\\
		&= \sigma \int_0^t \sum_{z \sim w} \cond_{\{z,w\}} \left(\bar \scale_w-\bar \scale_z \right)^2 \widehat \Pr_{\xi=\delta_x+\delta_y}\left(\xi_s(z)=1\ \text{and}\ \xi_s(w)=1\right)\dd s\ .	
		\end{align}
		If, additionally, we impose
		\begin{align*}
		\alpha_x=\alpha_L=\alpha_R\qquad \text{and}\qquad 
		\cond_L=\cond_R=1\qquad\text{and}\qquad
		\cond_{\{x,y\}}=\1_{\{|x-y|=1\}}\ ,
		\end{align*} for all $x, y \in V$, we further get (cf.\ \eqref{eq:linear_profile})
		\begin{equation}\label{eq:two_particle_duality_formula2}
		\E_\mu\left[D^\cl(\delta_x+\delta_y,\eta_t) \right]- \bar \scale_x \bar \scale_y=\sigma \frac{\left(\scale_L-\scale_R\right)^2}{\left(N+1\right)^2}  \int_0^t \widehat \Pr_{\xi=\delta_x+\delta_y}\left(\sum_{z=1}^{N-1} \xi_s(z)\xi_s(z+1) = 1\right)\dd s\ .
		\end{equation}
	\end{enumerate}
\end{remark}

\section{Orthogonal dualities}\label{section:orthogonal_duality}

By orthogonal dualities we refer to a specific subclass of  duality functions $D(\xi,\eta)$ in the form \eqref{eq:duality_factorized}. This subclass consists of jointly factorized functions whose each \textquotedblleft bulk\textquotedblright\ single-site duality function
\begin{equation*}
(k,n) \in \N_0 \times \N_0 \mapsto d_x(k,n) \in \R
\end{equation*} is a family of  polynomials in the $n$-variables and orthogonal  w.r.t.\ a suitable probability measure $\nu_x$ on $\N_0$, i.e.\ for all $k, \ell \in \N_0$,
\begin{align*}
\sum_{n=0}^\infty d_x(k,n)\, d_x(\ell,n)\, \nu_x(n) = \1_{\{k=\ell\}} \| d_x(k,\cdot)\|^2_{L^2(\nu_x)}\ .
\end{align*}
Orthogonal duality functions for exclusion, inclusion and independent particle systems with no interaction with reservoirs  have been first introduced in \cite{franceschini_stochastic_2017} by direct computations and then characterized in \cite{redig_factorized_2018} through generating function techniques. There, the dual particle system has the same law of the original particle system; therefore, orthogonal dualities are actually self-dualities. Moreover, for each $\sigma \in \{-1,0,1\}$,	 these jointly factorized orthogonal dualities  consist of products of hypergeometric functions of the following two types: either
\begin{align}\label{eq:hypergeometric_1}
\pFq{2}{0}{-k,-n}{-}{-u}:=&\   \sum_{\ell=0}^{k}\binom{k}{\ell} \left(\frac{n!}{(n-\ell)!}\1_{\{\ell \leq n\}}\right) u^\ell
\end{align}
or
\begin{align}\label{eq:hypergeometric_2}
\pFq{2}{1}{-k,-n}{v}{u}	
&:=\   \sum_{\ell=0}^{k} \binom{k}{\ell}\left( \frac{\Gamma(v) }{\Gamma(v+\ell)}\frac{n!}{(n-\ell)!}\1_{\{\ell\leq n\}}\right)	u^\ell\ ,
\end{align}
with $k,n \in \N_0$ and $u, v	 \in \R.$ 
More specifically, these orthogonal  single-site self-duality functions are Kravchuk polynomials for $\SEP$ $(\sigma = -1)$, Charlier polynomials for $\IRW$ ($\sigma = 0$) and Meixner polynomials for $\SIP$ ($\sigma = 1$) (see e.g.\ \cite{koekoek_hypergeometric_2010}). It turns out that such single-site self-duality functions are  orthogonal families w.r.t.\  the   single-site marginals of the  stationary (actually reversible) product measures  of the corresponding particle system; in particular, Kravchuk polynomials are orthogonal w.r.t.\ Binomial distributions, Charlier polynomials w.r.t.\ Poisson distributions and Meixner polynomials w.r.t.\ Negative Binomial distributions. More precisely, because in this setting there exists a one-parameter family of stationary product measures for each of the three particle systems (see also Section \ref{section:measures} above), this corresponds to the existence of a one-parameter family of orthogonal duality functions.

This correspondence between orthogonal duality functions and stationary measures may suggest that,  knowing a stationary measure of a particle system, an orthogonal family of observables of this system would correspond, in general, to duality functions.  This program, however, besides not being generally verifiable, does not apply to the case of particle systems in contact with reservoirs, for which the non-equilibrium stationary measures are, generally speaking, not in product form and not explicitly known (see also Section \ref{section:non-equilibrium}).

Nevertheless, from an algebraic point of view (see e.g.\ \cite{giardina_duality_2007}), new duality relations may be generated from the knowledge of a  duality relation and a \emph{symmetry} of one of the two generators involved in the duality relation. In brief, given the following duality relation
\begin{align*}
\widehat{\mathcal L}_\eft D(\xi,\eta)= \mathcal L_\ight D(\xi,\eta)
\end{align*}
for all $ \xi \in \widehat{\mathcal X}$, $\eta \in \mathcal X$,
and a symmetry $\widehat{\mathcal K}$   for the generator $\widehat{\mathcal L}$, i.e., for all $f: \widehat{\mathcal X}\to \R$ and $\xi \in \widehat{\mathcal X}$,
\begin{align}\label{eq:commutation_relation}
\widehat{\mathcal K}\, \widehat{\mathcal L} f(\xi) = \widehat{\mathcal L}\, \widehat{\mathcal K} f(\xi)\ ,
\end{align}
then, if  $F(\widehat{\mathcal K})$ with $F: \R \to \R$ is a well-defined operator, the function $(F(\widehat{\mathcal K}))_\eft D(\xi,\eta)$ is a duality function between $\mathcal L$ and $\widehat{\mathcal L}$. Indeed, for all  $\eta \in \mathcal X$ and $\xi \in \widehat{\mathcal X}$, we have	
\begin{align*}
\widehat{\mathcal L}_\eft (F(\widehat{\mathcal K}))_\eft D(\xi,\eta)
&= (F(\widehat{\mathcal K}))_\eft \widehat{\mathcal L}_\eft D(\xi,\eta)\\
&= (F(\widehat{\mathcal K}))_\eft \mathcal L_\ight D(\xi,\eta)\\ &= \mathcal L_\ight(F(\widehat{\mathcal K}))_\eft D(\xi,\eta)\ .
\end{align*}

This latter approach is the one we follow here (Theorem \ref{proposition:orthogonal_duality} below) to recover  a  one-parameter family of orthogonal duality functions  for boundary driven particle systems. Its proof combines two ingredients: first, as already proved in \cite{carinci_consistent_2019}, we observe that the so-called \emph{annihilation operator on $V \cup \{L,R\}$} given, for all $f : \widehat{\mathcal X} \to \R$, by	
\begin{align}\label{eq:K}
\widehat{\mathcal K}f(\xi) &
= \widehat{\mathcal K}^\bulk f(\xi)+ \widehat{\mathcal K}^{L,R}f(\xi)\ ,
\end{align}
where
\begin{align*}
\widehat{\mathcal K}^\bulk f(\xi) =  \sum_{x \in V} \widehat{\mathcal K}_x f(\xi) = \sum_{x \in V} \xi(x)\, f(\xi-\delta_x)
\end{align*}
and
\begin{align*}
\widehat{\mathcal K}^{L,R}f(\xi) = \widehat{\mathcal K}_L f(\xi)+\widehat{\mathcal K}_R f(\xi)
=  \xi(L)\, f(\xi-\delta_L) + \xi(R)\, f(\xi-\delta_R)\ ,
\end{align*}
is a symmetry for the generator $\widehat{\mathcal L}$  associated to the particle systems with purely absorbing reservoirs and defined in \eqref{eq:generator_full_dual}. Then, we obtain the candidate orthogonal dualities by acting with  suitable exponential functions of this symmetry $\widehat{\mathcal K}$ on the classical duality functions appearing in Proposition \ref{proposition:classical_duality}. We recall that in \eqref{eq:single-site_orthogonal_LR} below, the convention $0^0:=1$ holds.

\begin{theorem}[\textsc{orthogonal duality functions}]\label{proposition:orthogonal_duality}
	For each choice of $\sigma \in \{-1,0,1\}$, let $\mathcal L$ and $\widehat{\mathcal L}$ be the infinitesimal generators given in \eqref{eq:generator_full} and \eqref{eq:generator_full_dual}, respectively, associated to the particle systems $\{\eta_t: t \geq 0 \}$ and $\{\xi_t: t \geq 0 \}$. Then the duality relations in \eqref{eq:duality_relation_expectations} and \eqref{eq:duality_relation_generators} hold with the duality functions $D^\oo_\scale : \widehat{\mathcal X}\times \mathcal X \to \R$ defined, for all $\scale \in \Scale$, as follows: for all configurations $\eta \in \mathcal X$ and $\xi \in \widehat{\mathcal X}$,
\begin{align*}
D^\oo_\scale(\xi,\eta)=d^\oo_{L,\scale}(\xi(L))\times\left( \prod_{x\in V} d^\oo_{x,\scale}(\xi(x),\eta(x))\right)\times d^\oo_{R,\scale}(\xi(R))
\end{align*}	
	where, for all $x\in V$ and $k,n\in \N_0$,
	\begin{align}
	d^\oo_{x,\scale}(k,n)\ &=\ (-\scale)^k \times \begin{dcases} \pFq{2}{1}{-k,-n}{-\alpha_x}{\frac{1}{\scale}} & \sigma= -1\\[.15cm]
	\pFq{2}{0}{-k,-n}{-}{-\frac{1}{\scale \alpha_x}} &\sigma = 0\\[.15cm]
	\pFq{2}{1}{-k,-n}{\alpha_x}{-\frac{1}{\scale}} &\sigma=1\ ,
	\end{dcases}
	\label{eq:single-site_orthogonal}
	\end{align}
	and
	\begin{align}\label{eq:single-site_orthogonal_LR}
	&d^\oo_{L,\scale}(k) = \left(\scale_L-\scale \right)^k \qquad \text{and}\qquad
	d^\oo_{R,\scale}(k) = \left(\scale_R-\scale \right)^k\  .
	\end{align}
\end{theorem}
\begin{proof} We start with the observation that, for each $\sigma \in \{-1,0,1\}$,  the commutation relation \eqref{eq:commutation_relation} between the annihilation operator $\widehat{\mathcal K}$ in \eqref{eq:K} and the generator $\widehat{\mathcal L}$ \eqref{eq:generator_full_dual} holds (for a detailed proof, we refer to  e.g.\  \cite[Section 5]{carinci_consistent_2019}).
	
	As a consequence, for all $\scale \in \Scale$, the following function
	\begin{align}\label{eq:orth_duality_exponential}
	(e^{-\scale \widehat{\mathcal K}})_\eft D^\cl(\xi,\eta)
	\end{align}
	is a duality function between $\mathcal L$ and $\widehat{\mathcal L}$. In particular, recalling the definitions of single-site classical duality functions in \eqref{eq:single-site_classical}--\eqref{eq:single_site_classical_LR} and hypergeometric functions in \eqref{eq:hypergeometric_1}--\eqref{eq:hypergeometric_2},  due to the factorized form of both symmetry $e^{-\scale \widehat{\mathcal K}}$ and classical duality function, the combination of
	\begin{align*}
	(e^{-\scale \widehat{\mathcal K}_L}) d^\cl_L(k) &= \sum_{\ell=0}^k  \binom{k}{\ell} d^\cl_L(\ell)\, (-\scale)^{k-\ell}  = (\scale_L-\scale)^k\\
	(e^{-\scale \widehat{\mathcal K}_R}) d^\cl_R(k) &= \sum_{\ell=0}^k  \binom{k}{\ell} d^\cl_R(\ell)\, (-\scale)^{k-\ell}= (\scale_R-\scale)^k\
	\end{align*}
	and
	\begin{equation}\label{eq:binomial_power}
	(e^{-\scale \widehat{\mathcal K}_x})_\eft d^\cl_x(k,n) = \sum_{\ell=0}^k \binom{k}{\ell} d^\cl_x(\ell,n)\,  (-\scale)^{k-\ell} = (-\scale)^k  \begin{dcases} \pFq{2}{1}{-k,-n}{-\alpha_x}{\frac{1}{\scale}} & \sigma= -1\\[.15cm]
	\pFq{2}{0}{-k,-n}{-}{-\frac{1}{\scale \alpha_x}} &\sigma = 0\\[.15cm]
	\pFq{2}{1}{-k,-n}{\alpha_x}{-\frac{1}{\scale}} &\sigma=1\ ,
	\end{dcases}
	\end{equation}
	for all $x \in V$, concludes the proof.
\end{proof}

The above method to derive the orthogonal duality functions may be summarized as consisting in the application on the classical duality functions of a suitable symmetry on the \textquotedblleft left\textquotedblright\ dual variables $\xi$. This approach differs from all those previously employed in the context of closed systems: e.g.,  \cite{franceschini_stochastic_2017} is based on solving suitable recurrence relations, \cite{redig_factorized_2018} on computing generating functions, while \cite{carinci2019orthogonal} on acting with suitable unitary symmetries on the \textquotedblleft right\textquotedblright\ variables $\eta$. The main advantage of our method is that it works in both contexts of closed and open systems with no substantial alteration, since the annihilation operator is a commutator of the dual generator in both situations.

\begin{remark}
	To provide the reader with a further interpretation of  orthogonal dualities, we note that the  following formula connecting orthogonal and classical dualities is reminiscent of the Newton binomial formula:
	\begin{align}\label{eq:fake_binomial}
		d^\oo_{x,\scale}(k,n) =\sum_{\ell=0}^k \binom{k}{\ell} d^\cl_x(\ell,n)\, (-\scale)^{k-\ell}\ .
	\end{align}
In particular, setting $\scale=0$ and recalling the convention $0^0:=1$,
\begin{align}\label{eq:classical-orthogonal_theta=0}
 d^\oo_{x,\scale=0}(k,n)= \sum_{\ell=0}^k \binom{k}{\ell} d^\cl_x(\ell,n)\, (-0)^{k-\ell}= 	d^\cl_x(k,n)\ ,
\end{align}
i.e.,    the classical duality functions, $D^\cl(\xi,\eta)$, may be seen as a particular instance of the orthogonal duality functions if the scale parameter $\scale \in \Scale$ is set equal to zero, $D^\oo_{\scale=0}(\xi,\eta)$ (cf.\ \cite[\S4.1.1 \& \S4.1.2]{redig_factorized_2018}).
\end{remark}

\begin{remark}[\textsc{orthogonality relations}]\label{remark:orthogonal}
	In general,   the orthogonal duality functions of Theorem \ref{proposition:orthogonal_duality} are \emph{not} orthogonal w.r.t.\ the stationary measure of the particle dynamics in non-equilibrium. In fact, for each choice of $\sigma \in \{-1,0,1\}$ and $\scale \in \Scale$, the  orthogonal duality function $D^\oo_\scale(\xi,\eta)$  gives rise to an  orthogonal basis $\{e_\xi : \xi \in \widehat{\mathcal Y}\}$ of $L^2(\mathcal X, \mu_\scale)$, where $\mu_\scale$ is given in \eqref{eq:muscale},
	\begin{align}\label{eq:Y}
	e_\xi := D^\oo_\scale(\xi,\cdot)\qquad \text{and}\qquad \widehat{\mathcal Y} :=\{\xi \in \widehat{\mathcal X}: \xi(L)=\xi(R)=0 \}\ .
	\end{align}

	In equilibrium, i.e.\ $\scale_L = \scale_R = \scale \in \Scale$,  we have seen (see Section \ref{section:measures}) that the measure $\mu_\scale$ is stationary for the particle system $\{\eta_t: t\geq 0 \}$. In  non-equilibrium, i.e.\ $\scale_L \neq \scale_R$,  $\mu_\scale$ fails to be stationary. Nevertheless, the aforementioned orthogonality relations still hold in both contexts, regardless of the stationarity of $\mu_\scale$.

\end{remark}

As an immediate consequence of Theorem \ref{proposition:orthogonal_duality}, we can compute the following expectations of the orthogonal duality functions.
\begin{proposition}
	Let $\bubu\in \R$ such that
	\begin{equation}\label{eq:beta_theta}
	\scale:=\scale_R+ \bubu(\scale_L-\scale_R) \in \Scale\ .
	\end{equation}
	Then, for all $t \geq 0$ and  for all configurations $\xi \in \widehat{\mathcal X}$, we have
	\begin{align}\label{eq: expectation at time t of orthogonal duality}
	\E_{\mu_{\scale}}\left[D^\oo_\scale(\xi,\eta_t) \right] = (\scale_L-\scale_R)^{|\xi|} \phi_{t,\bubu}(\xi)\ ,
	\end{align}
	where  $\mu_\scale$ is the product measure (cf.\	 \eqref{eq:muscale}) with scale parameter $\scale = \scale_R+\bubu\left(\scale_L-\scale_R \right)$  and
	\begin{align*}
	\phi_{t,\bubu}(\xi)
	:=\left(-\bubu\right)^{\abs{\xi}} \widehat \E_\xi \left[   \left(\frac{\bubu-1}{\bubu}\right)^{\xi_t(L)} \boldsymbol 1_{\{\xi_t(L)+\xi_t(R)=\abs{\xi}\}}  \right]\ 	.
	\end{align*}
	Moreover,  for all configurations $\xi \in \widehat{\mathcal X}$, we have
	\begin{align}\label{eq: stationary expectation ort duality}
	\E_{\mu_{\scale_L,\scale_R}}\left[D^\oo_\scale(\xi,\eta) \right] = \left(\scale_L-\scale_R\right)^{|\xi|} \phi_{\bubu}(\xi)\ ,
	\end{align}
	where
	\begin{align*}
	\phi_{\bubu}(\xi)
	:=\left(-\bubu\right)^{\abs{\xi}} \widehat \E_\xi \left[   \left(\frac{\bubu-1}{\bubu}\right)^{\xi_\infty(L)}\right].
	\end{align*}
	In particular, $\phi_{\bubu,t}$ and $\phi_{\bubu}$ 	   do not depend on neither $\scale_L$ nor $\scale_R$, but only on $\bubu$, $\sigma \in \{-1,0,1\}$  and the underlying geometry of the system.
\end{proposition}
\begin{proof}
	As a consequence of duality (Theorem \ref{proposition:orthogonal_duality}), we obtain
	\begin{align}\label{eq:lastline}\nonumber
	\E_{\mu_\scale}\left[D^\oo_\scale(\xi,\eta_t) \right]
	&=\ \widehat \E_\xi\left[\E_{\mu_\scale}\left[D^\oo_\scale(\xi_t,\eta) \right] \right]\\
	\nonumber
	&=\ \widehat \E_\xi\left[\left(\scale_L-\scale \right)^{\xi_t(L)}\left(\scale_R-\scale \right)^{\xi_t(R)}	\1_{\{\xi_t(L)+\xi_t(R)=|\xi|\}} \right]\\
	\nonumber
	&=\ \widehat \E_\xi\left[\left(\scale_L-\scale \right)^{\xi_t(L)}\left(\scale_R-\scale \right)^{|\xi|-\xi_t(L)}	\1_{\{\xi_t(L)+\xi_t(R)=|\xi|\}} \right]\\
	&=\ \left(\scale_R-\scale \right)^{|\xi|}  \widehat \E_\xi\left[\left(\scale_L-\scale \right)^{\xi_t(L)}\left(\scale_R-\scale \right)^{-\xi_t(L)}	\1_{\{\xi_t(L)+\xi_t(R)=|\xi|\}} \right]\ ,
	\end{align}
where in the second identity we have used orthogonality of the single-site duality functions $d^\oo_{x,\scale}(k,\cdot) $ w.r.t.\ the marginal $\nu_{x,\scale}$ (see also Remark \ref{remark:orthogonal}) and the observation that
\begin{align*}
	d^\oo_{x,\scale}(0,\cdot)\equiv 1\ ,\quad x \in V\ .
\end{align*} 
	Inserting $\scale = \scale_R+\bubu (\scale_L-\scale_R)$ (cf.\ \eqref{eq:beta_theta}) in the last line of \eqref{eq:lastline}, we get \eqref{eq: expectation at time t of orthogonal duality}. By sending $t \to \infty$, the uniqueness of the stationary measure yields \eqref{eq: stationary expectation ort duality}.
\end{proof}

\begin{remark}
	For the choice $\bubu=\frac{1}{2}$ and, thus, $\scale=\frac{\scale_L+\scale_R}{2}$, \eqref{eq: expectation at time t of orthogonal duality} and \eqref{eq: stationary expectation ort duality}	further simplify as
	\begin{align}
	\E_{\mu_{\scale}}\left[D^\oo_\scale(\xi,\eta_t) \right] = \left(\frac{\scale_L-\scale_R}{2}\right)^{|\xi|} \widehat \E_\xi \left[   \left(-1\right)^{|\xi|-\xi_t(L)} \boldsymbol 1_{\{\xi_t(L)+\xi_t(R)=\abs{\xi}\}} \right]
	\end{align}
	and
	\begin{align}
	\E_{\mu_{\scale_L,\scale_R}}\left[D^\oo_\scale(\xi,\eta) \right] =\left(\frac{\scale_L-\scale_R}{2}\right)^{|\xi|}  \widehat \E_\xi \left[   \left(-1\right)^{|\xi|-\xi_\infty(L)}\right]\ .
	\end{align}
\end{remark}

\section{Higher order correlations in	 non-equilibrium}\label{section:n-point}
In this section, we study higher order space correlations for the non-equilibrium stationary measures presented in Section \ref{section:measures}. In particular, we show in Theorem \ref{theorem:n_correlation} below, by using the orthogonal duality functions of Section \ref{section:orthogonal_duality}, that the $n$-point correlation functions in non-equilibrium may be factorized into a first term, namely $(\scale_L-\scale_R)^n$, and a second term, which we call  $\psi$ and which is independent  of the values $\scale_L$ and $\scale_R$. This result may be seen as a higher order  generalization of the  decomposition obtained for the simple symmetric exclusion process in \cite[Eqs.\ (2.3)--(2.8)]{derrida_entropy_2007}. There the authors exploit the matrix formulation of the non-equilibrium stationary measure to recover the explicit expression for the first, second and third order correlation functions.

While the coefficients $\psi$ in \eqref{eq:psi_explicit} for the case of independent random walkers are identically zero (see item \ref{it:main_1} after Theorem \ref{theorem:n_correlation} below), for the interacting case ($\sigma \in \{-1,1\}$) they are expressed in terms of absorption probabilities of both  interacting and independent dual particles. These absorption probabilities -- apart from some special instances, see e.g.\ \cite{derrida_entropy_2007} and \cite[Section 6.1]{carinci_duality_2013-1} -- are not explicitly known. Nonetheless, Theorem \ref{theorem:n_correlation}  -- and the related Theorem \ref{theorem:main} -- highlight the common structure of the higher order correlations for all three particle systems considered in this paper. In particular, this common structure arises for all  values of the parameters $\scale_L$ and $\scale_R \in \Scale$ and with all disorders $(\boldsymbol \cond, \boldsymbol \alpha)$ and parameters $\{\cond_L,\cond_R,\alpha_L,\alpha_R \}$ as in \eqref{eq:environment1}--\eqref{eq:environment2}. Moreover, along the same lines, we show that all higher order space correlations at any finite time $t > 0$ for the particle system started from suitable product measures exhibit the same structure. This is the content of Theorem \ref{theorem:main} in Section \ref{section:correlations_time_t} below. In fact, we derive Theorem \ref{theorem:n_correlation} on the structure of stationary correlations from the more general result stated in Theorem \ref{theorem:main}, whose proof is deferred to Section \ref{section:proof_main}.

\subsection{Stationary non-equilibrium correlations and cumulants}

For each choice of $\sigma \in \{-1,0,1\}$, we recall that $\mu_{\scale_L,\scale_R}$ denotes the non-equilibrium stationary measure of the particle system $\{\eta_t: t \geq 0 \}$ with generator $\mathcal L$ given in \eqref{eq:generator_full}. Moreover, let us recall the definition of $\{\bar \scale_x: x \in V \}$ in \eqref{eq:scalex} and introduce the following ordering of dual configurations: for all $\xi \in \widehat{\mathcal X}$,
\begin{align}\label{eq:ordering}
\zeta \leq \xi\quad \text{if and only if}\quad \zeta \in \widehat{\mathcal X}\quad \text{and}\quad\begin{split}
\zeta(L)&\leq \xi(L)\ ,\quad
\zeta(R)\leq \xi(R)\\
\zeta(x)&\leq\xi(x)\ ,\quad \text{for all}\ x \in V	\ .
\end{split}
\end{align}
Analogously, we say that $\zeta < \xi$ if $\zeta \leq \xi$ and at least  one of the inequalities in \eqref{eq:ordering} is strict. Finally, given $\xi, \zeta\in \widehat{\mathcal X}$, let $\xi\pm\zeta$ denote  the configuration  with $\xi(x)\pm\zeta(x)$ particles at site $x$, for all $x \in V \cup \{L,\R\}$, as long as $\xi\pm \zeta \in \widehat {\mathcal X}$.

In what follows, for all choices of $\sigma \in \{-1,0,1\}$,  $\widehat \Pr$ and $\widehat \E$ denote the law and expectation, respectively, of the dual process with either exclusion ($\sigma = -1$), inclusion ($\sigma =1$) or no interaction ($\sigma =0$), while we adopt $\widehat \Pr^\IRW$ and $\widehat \E^\IRW$ to refer to the law and corresponding expectation, respectively, of  the dual process consisting of non-interacting random walks ($\sigma =0$).
\begin{theorem}[\textsc{stationary correlation functions}]\label{theorem:n_correlation}
	For all $n \in \N$ with $n \leq |V|$	 and for all $x_1, \ldots, x_n \in V$ with $x_i \neq x_j$ if $i \neq j$, by setting
	\begin{equation*}\xi=\delta_{x_1}+\cdots+\delta_{x_n}\ ,\end{equation*}
	we have
	\begin{align}\label{eq:psi}
	\nonumber
	\E_{\mu_{\scale_L,\scale_R}}\left[\prod_{i=1}^n\left(\frac{\eta(x_i)}{\alpha_{x_i}} - \bar \scale_{x_i}\right)\right] &= \left(\scale_L-\scale_R \right)^n \psi(\xi)\\
	&=\left(\scale_L-\scale_R \right)^n \psi(\delta_{x_1}+\cdots+\delta_{x_n})\  ,
	\end{align}
	where
	\begin{equation}\label{eq:psi_explicit}
	\psi(\xi)= \sum_{\zeta \le \xi}	 (-1)^{|\xi|-|\zeta|}\, \widehat \P_{\xi-\zeta}^{\IRW}((\xi-\zeta)_\infty(L)=\abs{\xi-\zeta})\,  \widehat \P_\zeta (\zeta_\infty(L)=\abs{\zeta})\ .
	\end{equation}
	In particular, $\psi(\xi) \in \R$  and it  does not depend on neither $\scale_L$ nor $\scale_R$, but only on $\sigma \in \{-1,0,1\}$ and the underlying geometry (see Eqs.\ \eqref{eq:environment1}--\eqref{eq:environment2}) of the system.
\end{theorem}
As an immediate consequence we have the following corollary on the stationary non-equilibrium joint cumulants.
\begin{corollary}[\textsc{joint cumulants}]
	For all $n \in \N$ and $x_1,\ldots, x_n \in V$ with $x_i \neq x_j$ if $i \neq j$, let  $\kappa(\delta_{x_1}+\ldots+\delta_{x_n})$  denote the	 joint cumulant  of the random variables
	\begin{align*}
	\left\{\frac{\eta(x_i)}{\alpha_{x_i}}-\bar \scale_{x_i}  : x_1, \ldots, x_n \in V \right\}\ .
	\end{align*}
	Then, we have
	\begin{align*}
	\kappa(\delta_{x_1}+\cdots+\delta_{x_n}) = \left(\scale_L -\scale_R \right)^n \varphi(\delta_{x_1}+\cdots+\delta_{x_n})\ ,
	\end{align*}
	where $\varphi(\delta_{x_1}+\cdots+\delta_{x_n}) \in \R$ does not depend on neither $\scale_L$ nor $\scale_R$, but only on $\sigma \in \{-1,0,1\}$ and the underlying geometry of the system.
\end{corollary}
\begin{proof}
	After recalling that
	\begin{align*}
	\kappa(\delta_{x_1}+\cdots+\delta_{x_n}) = \sum_{\gamma \in \mathcal T} (|\gamma|-1)!\, (-1)^{|\gamma|-1} \prod_{U \in \gamma} \E_{\mu_{\scale_L,\scale_R}}\left[\prod_{y \in U} \left( \frac{\eta(y)}{\alpha_y}-\bar \scale_y \right)\right]\ ,
	\end{align*}
	where $\mathcal T=\mathcal T(\{x_1,\ldots, x_n\})$ denotes the set of partitions of $\{x_1,\ldots, x_n\} \subset V$, the result follows by \eqref{eq:psi} with $\varphi(\{x_1,\ldots, x_n\})$ given by
	\begin{align*}
	\varphi(\delta_{x_1}+\cdots+\delta_{x_n}) = \sum_{\gamma \in \mathcal T} (|\gamma|-1)!\, (-1)^{|\gamma|-1} \prod_{U \in \gamma} \psi(U) \ ,
	\end{align*}
	where  $\psi(U):= \psi(\sum_{x \in U}\delta_x)$.	\end{proof}
\subsubsection{Properties of the function $\psi$}
We collect below some further properties of the coefficients $\psi$ in \eqref{eq:psi}:
\begin{enumerate}[label={\normalfont (\alph*)}]
	\item \label{it:main_2}For all $\sigma \in \{-1,0,1\}$, if $\abs{\xi}=0$, i.e. the dual  configuration is empty, then $\psi(\xi)=1$.
	\item \label{it:main_1}For $\sigma =0$, $\psi(\xi)=0$ for all $\xi \in \widehat {\mathcal X}$ such that $|\xi|\ge 1$.
	\item \label{it:main_2}For all $\sigma \in \{-1,0,1\}$ and for all $x \in V$, $\psi(\delta_x)=0$.
	\item \label{it:main_3}If $\sigma \in \{-1,1\}$ and $\scale_L \neq \scale_R$,  as a consequence of Theorem \ref{theorem:sign_two_point} and  $\left(\scale_L-\scale_R\right)^2 >0$, $\psi(\delta_x+\delta_y)$ is negative for $\sigma = -1$ and positive for $\sigma = 1$ for all $x, y \in V$.
	\item \label{it:main_4}Because $\psi(\delta_{x_1}+\cdots+\delta_{x_n})$ depends only on the underlying geometry of the system and not on $\scale_L, \scale_R$, exchanging the role of $\scale_L$ and $\scale_R$ does not affect the value of the stationary $n$-point correlation functions if $n \in \N$ is even, while it involves only a change of sign if $n \in \N$ is odd. More precisely, for all $n \in \N$ and $x_1,\ldots, x_n \in V$,
	\begin{align*}
	\E_{\mu_{\scale_L,\scale_R}}\left[\prod_{i=1}^n\left(\frac{\eta(x_i)}{\alpha_{x_i}}  -\bar \scale_{x_i}\right)\right]
	= \left(-1\right)^n\,
	\E_{\mu_{\scale_R,\scale_L}}\left[\prod_{i=1}^n\left(\frac{\eta(x_i)}{\alpha_{x_i}} -  \bar \scale_{x_i}\right)\right]\ .
	\end{align*}
		\item \label{it:main_5}As we will see in the course of the next section \ref{section:correlations_time_t}, $\psi(\xi)$   in \eqref{eq:psi}--\eqref{eq:psi_explicit} can be defined for any $\xi \in \widehat{\mathcal X}$ and equivalently expressed  in terms of a parameter $\bubu\in \R$. More precisely, given $\xi \in\widehat{\mathcal X}$ and $\bubu\in\R$, we have
	\begin{align}\label{eq:psibeta}
	\psi(\xi)=\sum_{\zeta\le \xi}   (-1)^{\abs{\xi}-\abs{\zeta}}\left (\prod_{x\in V} \binom{\xi(x)}{\zeta(x)} (\widehat p_{\infty}(\delta_x,\delta_L)-\bubu)^{\xi(x)-\zeta(x)} \right)\widehat \E_\zeta\left[  (1-\bubu)^{\zeta_\infty(L)}   (-\bubu)^{\zeta_\infty(R)}    \right]\ .
	\end{align}
	
	Notice that, by setting
	$\xi=\delta_{x_1}+\cdots+\delta_{x_n}$  with $x_i \neq x_j$ if $i \neq j$, all the binomial coefficients in \eqref{eq:psibeta} are equal to one. The choice $\bubu=0$ corresponds then to the expression on the l.h.s. of \eqref{eq:psi}, while choosing $\bubu=1$ leads to
	
	\begin{equation}\label{eq:psi2}
	\psi(\xi)=\sum_{ \zeta\le \xi}(-1)^{|\zeta|}\ \widehat \P_{\xi-\zeta}^{\IRW}((\xi-\zeta)_\infty(R)=\abs{\xi-\zeta})\, \widehat \P_\zeta (\zeta_\infty(R)=\abs{\zeta})\ .
	\end{equation}
	In particular, since $\psi(\xi)$ does not depend on $\bubu$, we have that
	\begin{equation}\label{eq: derivative of psi}
	\frac{\text{d}\psi(\xi)}{\text{d} \bubu}  = 0\ 	,
	\end{equation}
	which is an equation giving information on the absorption probabilities.
	If we consider, for instance, the case $\xi=\delta_x+\delta_y$ with $x\neq y$,  \eqref{eq:psibeta} and \eqref{eq: derivative of psi} yield	
	\begin{align}
	2\,\widehat \P_{\xi=\delta_x+\delta_y}(\xi_\infty(L) =2)+\widehat \P_{\xi=\delta_x+\delta_y}(\xi_\infty(L) =1)=\widehat p_{\infty}(\delta_x,\delta_L)+\widehat p_{\infty}(\delta_y,\delta_L)\ ,
	\end{align}
	which corresponds to the recursive relation found in \cite[Proposition 5.1]{carinci_consistent_2019}.
	More generally, by matching the two expressions of $\psi(\xi)$ for $\xi=\delta_{x_1}+\cdots+\delta_{x_n}$  with $x_i \neq x_j$ if $i \neq j$, in \eqref{eq:psi_explicit} and \eqref{eq:psi2},  the relation that we find is
	\begin{align*}
	&\ \widehat \Pr_\xi\left(\xi_\infty(L)=|\xi| \right) - \left(-1\right)^{|\xi|}\widehat \Pr_\xi\left(\xi_\infty(R)=|\xi| \right)\\[.15cm]
	&= \sum_{\zeta< \xi}\left\{\begin{array}{l}   \widehat \P_{\xi-\zeta}^{\IRW}((\xi-\zeta)_\infty(R)=\abs{\xi-\zeta})\, \widehat \P_\zeta (\zeta_\infty(R)=\abs{\zeta})\\[.2cm]
	-(-1)^{|\xi|}\ \widehat \P_{\xi-\zeta}^{\IRW}((\xi-\zeta)_\infty(L)=\abs{\xi-\zeta})\,	 \widehat \P_\zeta (\zeta_\infty(L)=\abs{\zeta})\ \end{array}\right\}\ .
	\end{align*}
	In other words, the above equation relates the probabilities of having all $|\xi|$ dual particles absorbed at the same end with a linear combination of analogous probabilities for systems with a strictly smaller number of particles.

\end{enumerate}

\

\subsection{Correlations at finite times and proof of Theorem \ref{theorem:n_correlation}}\label{section:correlations_time_t}

Theorem \ref{theorem:n_correlation} follows from a more general result. This is the content of Theorem \ref{theorem:main} below. There, we show that a decomposition reminiscent of that in \eqref{eq:psi} holds also for expectations at some fixed positive time of generalizations of the $n$-point correlation functions of Theorem \ref{theorem:n_correlation} when the particle system starts from a suitable product measure.
The aforementioned generalizations of the correlation functions are constructed by suitably recombining the orthogonal duality functions of Section \ref{section:orthogonal_duality} so to obtain a family of  functions   orthogonal w.r.t.\ what we call \textquotedblleft interpolating product measures\textquotedblright\   given in the following definition.

\begin{definition}[\textsc{interpolating product measures}]
	We call  \emph{interpolating product measure}  with interpolating parameters
	\begin{equation}\label{eq:betaN}
	\boldsymbol \beta= \{\beta_x: x \in V \}
	\end{equation} the measure given by
	\begin{align}\label{eq:muLRbeta}
	\mu_{\scale_L,\scale_R,\boldsymbol \beta} := \otimes_{x \in V}\, \nu_{x,\scale_x}\ ,
	\end{align}
	with
	\begin{align}\label{eq:scalexx}
	\scale_x:=\scale_R+\beta_x(\scale_L-\scale_R)\,\ ,
	\end{align}
	where the marginals  $\{\nu_{x, \scale}: x \in V\}$ appearing in \eqref{eq:muLRbeta} are those given in \eqref{eq:marginals} and $\boldsymbol \beta $ in \eqref{eq:betaN}--\eqref{eq:scalexx} is chosen such that, for each choice of $\sigma \in \{-1,0,1\}$, the product measure $\mu_{\scale_L,\scale_R,\boldsymbol \beta}$ is a probability measure, i.e., for all $x \in V$, the following conditions hold:
	\begin{align}\label{eq:conditions_beta}
	\beta_x \in \R\qquad \text{and}\qquad \scale_x = \scale_R + \beta_x(\scale_L-\scale_R) \in \Scale\ .
	\end{align}
\end{definition}
In particular, if we choose
\begin{align*}
\beta_x= \widehat p_\infty(\delta_x,\delta_L)=: \bar \beta_x\ ,\quad x \in V\ ,
\end{align*}
as corresponding interpolating product measure we recover the local equilibrium product measure $\mu_{\bar {\boldsymbol \scale}}$ (Definition \ref{def: local equilibrium product measure}):
\begin{align}\label{eq:interpolating_local_equilibrium_product}
\mu_{\scale_L,\scale_R,\bar {\boldsymbol\beta}}= \mu_{\bar{\boldsymbol \scale}}\ .	
\end{align}

Let us now introduce what we call the \textquotedblleft interpolating orthogonal functions\textquotedblright.	
\begin{definition}[\textsc{interpolating orthogonal functions}]
	Recalling the definition of orthogonal polynomial dualities in \eqref{eq:single-site_orthogonal}--\eqref{eq:single-site_orthogonal_LR} and the definition of interpolating parameters $\boldsymbol \beta$ in \eqref{eq:betaN}, we define the \emph{interpolating orthogonal function} with interpolating parameters $\boldsymbol \beta$  as follows:
	\begin{align}\label{eq:duality_improper}
	D^\oo_{\scale_L,\scale_R,\boldsymbol \beta}(\xi,\eta) := d^\oo_{L,\scale_L}(\xi(L)) \times \left(\prod_{x \in V} d^\oo_{x, \scale_x}(\xi(x),\eta(x))\right) \times d^\oo_{R,\scale_R}(\xi(R))\ ,
	\end{align}
	where the parameters $\{\scale_x: x \in V \}$ are defined in terms of $\scale_L$, $\scale_R$ and $\boldsymbol \beta$ as in \eqref{eq:scalexx}.
\end{definition}
In analogy with \eqref{eq:interpolating_local_equilibrium_product}, we define
\begin{align}
D^\oo_{\bar{\boldsymbol \scale}}(\xi,\eta) := D^\oo_{\scale_L,\scale_R,\bar{\boldsymbol \beta}}(\xi,\eta)
=d^\oo_{L,\scale_L}(\xi(L)) \times \left(\prod_{x \in V} d^\oo_{x,\bar \scale_x}(\xi(x),\eta(x))\right) \times d^\oo_{R,\scale_R}(\xi(R))\ .
\end{align}

\begin{remark}
	We note that, despite the analogy in  notation, in general these functions are \emph{not} duality functions for the particle system $\{\eta_t: t \geq 0 \}$, unless we assume the system to be at equilibrium, i.e.\ $\scale_L = \scale_R = \scale \in \Scale$. Only in  the latter case, $D^\oo_{\scale_L,\scale_R,\boldsymbol \beta}(\xi,\eta) = D^\oo_\scale(\xi,\eta)$ for all choices of $\boldsymbol \beta$.	
\end{remark}

With the definition \eqref{eq:duality_improper},  we have (cf.\ Remark \ref{remark:orthogonal}) that
\begin{align}\label{eq:zero_boundary}
D^\oo_{\scale_L,\scale_R,\boldsymbol \beta}(\xi,\cdot) = 0\ , \qquad \text{if}\quad \xi \in \widehat{\mathcal X}\setminus  \widehat{\mathcal Y}\ ,	
\end{align}
and that the family of functions
\begin{align*}
\left\{D^\oo_{\scale_L,\scale_R,\boldsymbol \beta}(\xi,\cdot) : \xi \in \widehat{\mathcal Y} \right\}
\end{align*}
is an orthogonal basis in $L^2(\mathcal X,  \mu_{\scale_L,\scale_R,\boldsymbol \beta})$. Now we are ready to state the main result of this section, whose Theorem \ref{theorem:n_correlation} is a particular instance.

\begin{theorem}\label{theorem:main} Let us consider two set of interpolating parameters
	\begin{align*}
	\boldsymbol \beta = \left\{\beta_x: x \in V\right\}\quad \text{and}\quad \boldsymbol \beta'=\left\{\beta_x': x \in V \right\}
	\end{align*}
	both satisfying \eqref{eq:scalexx}. Then,
	for all $\xi \in \widehat{\mathcal Y} \subset \widehat {\mathcal X}$ and  $t \geq 0$, we have
	\begin{align}\label{eq:psi_t beta N}
	\E_{\mu_{\scale_L,\scale_R,\boldsymbol \beta}}\left[D^\oo_{\scale_L,\scale_R,\boldsymbol \beta'}(\xi,\eta_t) \right] = \left(\scale_L-\scale_R\right)^{|\xi|} \psi_{t,\boldsymbol \beta,\boldsymbol \beta'}(\xi)\ ,
	\end{align}
	where
	\begin{align}\label{eq:psi_general}
	\psi_{t,\boldsymbol \beta,\boldsymbol \beta'}(\xi)
	:=\sum_{\zeta \le \xi}  \left(-1\right)^{|\xi|-|\zeta|}
	\left( \prod_{x \in V} \binom{\xi(x)}{\zeta(x)} (\beta_x')^{\xi(x)-\zeta(x)}\  \widehat \E_\zeta\left[ \boldsymbol 1_{\{\zeta_t(R)=0\}}\left( \prod_{x \in V} (\beta_x)^{\zeta_t(x)} \right)\right]\right)  \ ,
	\end{align}
	and  $\psi_{t,\boldsymbol \beta,\boldsymbol \beta'}(\xi)$ does not depend on neither $\scale_L$ nor $\scale_R$, but only on $\boldsymbol \beta$, $\boldsymbol \beta'$, $\sigma \in \{-1,0,1\}$ and the underlying geometry of the system. Moreover, by sending $t$ to infinity in \eqref{eq:psi_t beta N} we obtain, for all $\xi \in \widehat{\mathcal Y} \subset \widehat{\mathcal X}$,
	\begin{align}\label{eq:psi_infty_beta}
	\E_{\mu_{\scale_L,\scale_R}}\left[D^\oo_{\scale_L,\scale_R,\boldsymbol \beta'}(\xi,\eta) \right] = \left(\scale_L-\scale_R \right)^{|\xi|} \psi_{\boldsymbol \beta'}(\xi)\ ,
	\end{align}
	where
	\begin{align}\label{eq:psi_infty_beta2}
	\psi_{\boldsymbol \beta'}(\xi):= \sum_{\zeta \leq \xi} \left(-1 \right)^{|\xi|-|\zeta|} \left( \prod_{x \in V} \binom{\xi(x)}{\zeta(x)} (\beta_x')^{\xi(x)-\zeta(x)} \right)\widehat \P_\zeta\left[ \zeta_\infty(L)=\abs{\zeta} \right] \ .
	\end{align}
	Again, $\psi_{\boldsymbol \beta'}(\xi)$ is independent of $\scale_L$ and $\scale_R$.
\end{theorem}

\begin{remark}
	From the proof of Theorem \ref{theorem:main}, the results of the theorem extend to configurations $\xi \in \widehat {\mathcal X}\setminus \widehat{\mathcal Y}$ and,  by \eqref{eq:zero_boundary}, 	
	\begin{align}\label{eq:zero_boundary_psi}
	\psi_{t,\boldsymbol \beta,\boldsymbol \beta'}(\xi)=0\ ,\qquad \text{if}\quad \xi \in \widehat {\mathcal X}\setminus \widehat{\mathcal Y}\ .
	\end{align}
\end{remark}

Before moving to the next section, Section \ref{section:proof_main}, in which we provide the proof of Theorem \ref{theorem:main}, we show how this latter result implies Theorem \ref{theorem:n_correlation}.
\begin{proof}[Proof of Theorem~\ref{theorem:n_correlation}]
	Recall that,   by the definitions of hypergeometric functions \eqref{eq:hypergeometric_1}--\eqref{eq:hypergeometric_2} and of single-site orthogonal duality functions in \eqref{eq:single-site_orthogonal},   we have, for all $n \in \N$ and $\eta \in \mathcal X$,
	\begin{align}\label{eq:correlation_improper}
	D^\oo_{\bar{\boldsymbol \scale}}(\delta_{x_1}+\cdots+\delta_{x_n},\eta) = \prod_{i=1}^n \left(\frac{\eta(x_i)}{\alpha_{x_i}}-\bar \scale_i \right)
	\end{align}
	anytime $x_1,\ldots, x_n \in V$ with $x_i \neq x_j$ if $i \neq j$.
	By choosing for any $x\in V$, $\beta'_x=\widehat p_\infty(\delta_x,\delta_L)$, the result follows immediately from Theorem \ref{theorem:main}.
\end{proof}

\subsubsection{Probabilistic interpretation of the function $\psi$ }
Theorem \ref{theorem:n_correlation} may be seen as a particular instance of Theorem \ref{theorem:main} with the choice $t =\infty$, $\xi \in \widehat{\mathcal X}$ consisting of finitely many particles all sitting at different sites in the bulk and $\beta'_x=\widehat p_\infty(\delta_x,\delta_L)$  for every $x\in V$.  In fact, Theorem \ref{theorem:main} extends the relation \eqref{eq:psi} to all $\xi \in \widehat{\mathcal X}$, i.e.\
\begin{align}
\E_{\mu_{\scale_L,\scale_R}}\left[D^\oo_{\bar{\boldsymbol \scale}}(\xi,\eta) \right]= \left(\scale_L-\scale_R \right)^{|\xi|} \psi(\xi)\ ,
\end{align}
with, 
\begin{align}\label{eq: psi infinity general}
\psi(\xi):= \sum_{ \zeta\le \xi} \binom{\xi}{\zeta}\,   (-1)^{|\xi|-|\zeta|}\ \widehat \P_{\xi-\zeta}^{\IRW}((\xi-\zeta)_\infty(L)=\abs{\xi-\zeta})\, \widehat \P_\zeta\left( \zeta_\infty(L)=\abs{\zeta} \right) \ ,
\end{align}
where $\binom{\xi}{\zeta}:= \prod_{x \in V} \binom{\xi(x)}{\zeta(x)}$ and $\widehat \Pr^\IRW$ refers to the law  of  the dual process for $\sigma =0$, consisting of non-interacting random walks. 

In order to obtain a more probabilistic interpretation of \eqref{eq: psi infinity general}, we define

\begin{enumerate}[label={\normalfont (\alph*)}]
	\item the probability measure  $\gamma_\xi$   on $ \widehat {\mathcal X}$ given by 
	\begin{align}\label{eq:uniform_measure}
	\gamma_\xi(\zeta) =   \frac{\binom{\xi}{\zeta}}{2^{\abs{\xi}}}\, \1_{\{\zeta \leq \xi\}},\ 
	\end{align}
	i.e. the distribution of  uniformly chosen  sub-configuration  of $\xi$ (i.e. $\zeta\le \xi$);
	\item the function $\Psi_\xi:\widehat {\mathcal X}\to \R$ given by
	$$\Psi_\xi(\zeta):=\1_{\{\zeta \leq \xi\}} \ \widehat \P_{\xi-\zeta}^{\IRW}((\xi-\zeta)_\infty(L)=\abs{\xi-\zeta})\, \widehat \P_\zeta\left( \zeta_\infty(L)=\abs{\zeta} \right) \ ,$$
	i.e., the function that assigns to any  $\zeta\le \xi$ the probability that, in a system composed by the superposition of the configuration $\zeta$ of \emph{interacting dual particles} and the configuration $\xi-\zeta$ of \emph{independent dual random walks}, 
	independent between each other, all the particles are eventually absorbed at $L$.
\end{enumerate}
The function $\psi(\xi)$ in \eqref{eq: psi infinity general} can, then, be rewritten as follows:

$$ \psi(\xi)=2^{\abs{\xi}}\sum_{\zeta \in \widehat{\mathcal X}}  (-1)^{\abs{\xi-\zeta}}\ \Psi_\xi(\zeta)\, \gamma_\xi(\zeta)\ 	.$$

Similarly, for all $t \geq 0$, $\xi \in \widehat{\mathcal X}$ and for the special choice
\begin{equation*}
\boldsymbol \beta = \boldsymbol \beta'\quad \text{and}\quad \beta_x=\beta_x'=\widehat p_\infty(\delta_x,\delta_L)\ ,
\end{equation*}
the identity in \eqref{eq:psi_t beta N} yields, as a particular case, 
\begin{align}\label{eq:psi_t5}
\E_{\mu_{\bar{\boldsymbol \scale}}}\left[D^\oo_{\bar{\boldsymbol \scale}}(\xi,\eta_t) \right]=\left(\scale_L-\scale_R \right)^{|\xi|} \psi_{t}(\xi)\ ,
\end{align}
where
\begin{align}\label{eq:psi_t}
	\nonumber
\ 	\psi_t(\xi):=&\ \sum_{ \zeta\le \xi}(-1)^{\xi-\zeta}\, \widehat \P_{\xi-\zeta}^{\IRW}((|\xi|-|\zeta|)_\infty(L)=\abs{\xi-\zeta})\, \widehat \E_\zeta\left[\widehat\P_{\zeta_t}^{\text{IRW}} (\zeta_\infty(L)=\abs{\zeta})\right]\\
=&\ 2^{\abs{\xi}}\,  \sum_{\zeta \in\widehat{\mathcal X}} (-1)^{|\xi|-|\zeta|}  \Psi_{t,\xi}(\zeta)\, 	\gamma_\xi(\zeta)\     ,
\end{align}
where the integral in the last identity is w.r.t.\ the probability measure $\gamma_\xi$ defined in \eqref{eq:uniform_measure} and 
$$\Psi_{t,\xi}(\zeta):=\1_{\{\zeta \leq \xi\}}\widehat \P_{\xi-\zeta}^{\IRW}((\xi-\zeta)_\infty(L)=|\xi|-|\zeta|)\, \widehat \E_\zeta\left[\widehat\P_{\zeta_t}^{\text{IRW}} (\zeta_\infty(L)=|\zeta|)\right].$$

\subsection{Proof or Theorem \ref{theorem:main}}\label{section:proof_main}

We prove Theorem \ref{theorem:main} in two steps.

First we obtain a formula to relate the functions $D^\oo_{\scale_L,\scale_R,\boldsymbol \beta'}(\xi,\eta)$ in \eqref{eq:duality_improper} appearing in the statement of Proposition \ref{theorem:main} to the orthogonal duality functions $D^\oo_\scale(\xi,\eta)$ in Section \ref{section:orthogonal_duality}, for some $\scale \in \Scale$.

\begin{lemma}\label{lemma:replacement} For each choice of $\sigma \in \{-1,0,1\}$ and $\bubu\in \R$,  we define
	\begin{equation}\label{eq: scale beta}
	\scale:=\scale_R+\bubu(\scale_L-\scale_R)\  .
	\end{equation}
	Then,  for all configurations $\eta \in \mathcal X$ and $\xi \in \widehat {\mathcal X}$,	
	\begin{align*}
	D^\oo_{\scale_L,\scale_R,\boldsymbol \beta'}(\xi,\eta) =&\ \sum_{\zeta \leq \xi} \left(\scale_L-\scale_R\right)^{|\xi|-|\zeta|} \left(-1\right)^{|\xi|-|\zeta|} E_{\boldsymbol \beta',\bubu}(\zeta,\xi)\, D^\oo_\scale(\zeta,\eta)\ ,
	\end{align*}
	where  $E_{\boldsymbol \beta',\bubu}(\zeta,\xi)$ is defined as
	\begin{align}\label{eq:E}
	E_{\boldsymbol \beta',\bubu}(\zeta,\xi)	:=  E_{L,\bubu}(\zeta(L),\xi(L)) \times \left( \prod_{x \in V} E_{x,\beta_x',\bubu}(\zeta(x),\xi(x))\right) \times E_{R,\bubu}(\zeta(R),\xi(R))\ ,
	\end{align}
	where, for all $x \in V$,
	\begin{align*}
	E_{x,\beta_x',\bubu}(\ell,k) :=&\ \binom{k}{\ell} (\beta_x'-\bubu)^{k-\ell} \1_{\{\ell \leq k \}}\ ,
	\end{align*}
	and
	\begin{align*}
	E_{L,\bubu}(\ell,k) :=&\ \binom{k}{\ell} (1-\bubu)^{k-\ell} \1_{\{\ell \leq k \}}\\
	E_{R,\bubu}(\ell,k) :=&\ \binom{k}{\ell} (-\bubu)^{k-\ell} \1_{\{\ell \leq k \}}\ .
	\end{align*}
\end{lemma}
\begin{proof}
	By  definition of the orthogonal duality functions in Theorem \ref{proposition:orthogonal_duality} (see also \eqref{eq:orth_duality_exponential}) and of the functions $D^\oo_{\scale_L,\scale_R,\boldsymbol \beta'}$ in \eqref{eq:duality_improper}, we have
	\begin{align*}
	D^\oo_\scale = \left(e^{- \scale \widehat{\mathcal K}} \right)_\eft D^\cl
	\end{align*}
	and
	\begin{align*}
	D^\oo_{\scale_L,\scale_R,{\boldsymbol \beta'}} = \left(e^{-\scale_L \widehat{\mathcal K}_L-\left(\sum_{x \in V} \scale_x' \widehat{\mathcal K}_x\right)-\scale_R \widehat{\mathcal K}_R} \right)_{\eft} D^\cl\quad\ ,
	\end{align*}
	where
	\begin{equation*}
	\scale_x':= \scale_R + \beta_x'\left(\scale_L-\scale_R \right)\ ,\quad x \in V\ .
	\end{equation*}
	Next,  we get
	\begin{align*}
	D^\oo_{\scale_L,\scale_R,{\boldsymbol \beta'}} &=
	\left(e^{-\scale_L \widehat{\mathcal K}_L-\left(\sum_{x \in V}  \scale_x' \widehat{\mathcal K}_x\right)-\scale_R \widehat{\mathcal K}_R+\scale \widehat {\mathcal K}} \right)_{\eft} \left(e^{- \scale \widehat{\mathcal K}} \right)_\eft D^\cl\quad\\
	&=\ \left(e^{-(\scale_L-\scale) \widehat{\mathcal K}_L-\left(\sum_{x \in V} (\scale_x'-\scale) \widehat{\mathcal K}_x\right)-(\scale_R-\scale) \widehat{\mathcal K}_R} \right)_{\eft} D^\oo_\scale\ ,
	\end{align*}
	where the latter identity is a consequence of the fact that all the operators $\{\widehat{\mathcal K}_x: x \in V \} \cup \{\widehat{\mathcal K}_L, \widehat{\mathcal K}_R \}$ commute. The expressions in terms of $\left(\scale_L-\scale_R\right)$ of the  parameters $\{ \scale_x' : x \in V \}$ in \eqref{eq:conditions_beta} and  $\scale$ in \eqref{eq: scale beta}   yield the final result.	\end{proof}

Then, we derive an analogue of Theorem \ref{theorem:main} for the orthogonal duality functions.
\begin{lemma}\label{lemma:duality}
	For each choice of $\sigma \in \{-1,0,1\}$ and  $\bubu \in \R$ and $\scale \in \R$ as in \eqref{eq: scale beta} and such that $\scale \in \Scale$, we have, for all configurations $\zeta \in \widehat{\mathcal X}$,
	\begin{align}\label{eq:npoint_duality}
	\E_{\mu_{\scale_L,\scale_R,\boldsymbol \beta}}\left[D^\oo_\scale(\zeta,\eta_t) \right] = \left(\scale_L-\scale_R\right)^{|\zeta|} \phi_{t,\boldsymbol \beta,\bubu}(\zeta)\ ,
	\end{align}
	where  $\phi_{t,\boldsymbol \beta,\bubu}(\zeta) \in \R$ is defined as
	\begin{align}\label{eq:phi}
	\phi_{t,\boldsymbol \beta,\bubu}(\zeta)
	:= \widehat \E_\zeta \left[  (1-\bubu)^{\zeta_t(L)} \times \left(\prod_{x \in V} (\beta_x-\bubu)^{\zeta_t(x)} \right) \times (-\bubu)^{\zeta_t(R)}\right]
	\end{align} and, in particular, it does not depend on neither $\scale_L$ nor $\scale_R$, but only on $\boldsymbol \beta$, $\bubu$, $\sigma \in \{-1,0,1\}$ and the underlying geometry of the system.
\end{lemma}
\begin{proof} Recall the definition of $\mu_{\scale_L,\scale_R,\boldsymbol \beta}$ in \eqref{eq:muLRbeta} and of the scale parameters $\{\scale_x: x \in V \}$ in \eqref{eq:scalexx}.
	By duality (Theorem \ref{proposition:orthogonal_duality}), we have
	\begin{align*}
	\E_{\mu_{\scale_L,\scale_R,\boldsymbol \beta}}\left[D^\oo_\scale(\zeta,\eta_t) \right]&= \sum_{\zeta' \in \widehat{\mathcal X}} \widehat p_t(\zeta,\zeta')\, \E_{\mu_{\scale_L,\scale_R,\boldsymbol \beta}}\left[D^\oo_\scale(\zeta',\eta) \right]\\
	&= \sum_{\zeta' \in \widehat{\mathcal X}} \widehat p_t(\zeta,\zeta') \left\{(\scale_L-\scale)^{\zeta'(L)} \times \left(\prod_{x \in V} (\scale_x-\scale)^{\zeta'(x)} \right) \times (\scale_R-\scale)^{\zeta'(R)} \right\}\ ,	
	\end{align*}
	where this last identity is a consequence of
	\begin{align*}
	\sum_{n \in \N_0} d^\oo_{x,\scale}(k,n)\, \nu_{x,\scale_x}(n) = (\scale_x-\scale)^k
	\end{align*}
	for all $x \in V$ and $k \in \{0,\ldots, \alpha_x\}$ if $\sigma = -1$ and $k \in \N_0$ if $\sigma \in \{0,1\}$ (see e.g.\ \cite{redig_factorized_2018}).
	We obtain \eqref{eq:npoint_duality} with the function $\phi_{t,\boldsymbol \beta,\bubu}$ as in \eqref{eq:phi} by rewriting in terms of the parameters $\boldsymbol \beta$ and $\bubu$ the expression above between curly brackets.
\end{proof}

A combination of Lemma \ref{lemma:replacement} and Lemma \ref{lemma:duality} concludes the proof of Theorem \ref{theorem:main}. Indeed,
\begin{align*}
	\E_{\mu_{\scale_L,\scale_R,\boldsymbol \beta}}\left[D^\oo_{\scale_L,\scale_R,\boldsymbol \beta'}(\xi,\eta_t) \right]&= \sum_{\zeta \leq \xi} \left(\scale_L-\scale_R\right)^{|\xi|-|\zeta|} \left(-1\right)^{|\xi|-|\zeta|} E_{\boldsymbol \beta',\bubu}(\zeta,\xi)\, \E_{\mu_{\scale_L,\scale_R,\boldsymbol \beta}}\left[D^\oo_\scale(\zeta,\eta_t) \right]\\
&=  (\scale_L-\scale_R)^{|\xi|}\sum_{\zeta \leq \xi}  (-1)^{|\xi|-|\zeta|} E_{\boldsymbol \beta',\bubu}(\zeta,\xi)\,  \phi_{t,\boldsymbol \beta,\bubu}(\zeta)\ ,
\end{align*}
which yields \eqref{eq:psi_t beta N} with $\psi_{t,\boldsymbol \beta,\boldsymbol \beta'}(\xi)$ given by
\begin{align}\label{eq:psi_t2}
\psi_{t,\boldsymbol \beta,\boldsymbol \beta'}(\xi) =&\ \sum_{\zeta \in \widehat{\mathcal X}}  (-1)^{|\xi|-|\zeta|} E_{\boldsymbol \beta',\bubu}(\zeta,\xi)\,  \phi_{t,\boldsymbol \beta,\bubu}(\zeta)\ .
\end{align}
We note that, because the l.h.s.\ in \eqref{eq:psi_t beta N} and $(\scale_L-\scale_R)^{|\xi|}$ do not depend on the parameter $\bubu \in \R$,  the whole expression in \eqref{eq:psi_t2} is independent of $\bubu$, and in particular, we obtain \eqref{eq:psi_t beta N} for the choice $\bubu=0$. By passing to the limit as $t$ goes to infinity on both sides in \eqref{eq:psi_t beta N}, by uniqueness of the stationary measure $\mu_{\scale_L,\scale_R}$, we obtain \eqref{eq:psi_infty_beta}--\eqref{eq:psi_infty_beta2}.

\section{Exponential moments and generating functions}\label{section:exponential_generating_functions}
In this section we use the fact that the orthogonal dualities have explicit and simple generating functions in order to produce a formula for the joint moment generating function of the occupation variables in the non-equilibrium stationary state, in terms of the absorbing dual started from a random configuration $\xi$ of which the distribution is related to the  reservoir parameters. We recall that $\Scale=[0,1]$ if $\sigma= -1$ and $\Scale=[0,\infty)$ if $\sigma \in \{0,1\}$.

\begin{theorem}\label{theorem joint moment generating function}
	Let $\boldsymbol \lambda = \{\lambda_x: x \in V\} \in \R^N$ be such that, for all $x\in V$,
	\begin{equation}\label{eq: Lambda_x}
	\Lambda_x:=1+\frac{\lambda_x}{1+\sigma \lambda_x(1+\bar\scale_x)}\ge 0\ ,
	\end{equation}
	and
	\begin{align}\label{eq: K_x}
	\kappa_x := \frac{\lambda_x(\scale_L-\scale_R)}{1+\sigma \lambda_x\left(1-\left(\scale_L-\scale_R \right) \right)}\in \Scale\ .
	\end{align}
	Then,	 we have
	\begin{align}\label{eq:identity_frank}
	\E_{\mu_{\scale_L,\scale_R}}\left[\prod_{x \in V}\left(\Lambda_x \right)^{\eta(x)}\right]
	= \left(\prod_{x \in V} J_{\scale_L,\scale_R,\lambda_x}\right) \E_{\mu_{\boldsymbol k}}[\psi],
	\end{align}
	and, for all $t\geq 0$,
	\begin{align}\label{eq:identity_frank time t}
	\E_{\mu_{\bar{\boldsymbol \scale}}}\left[\prod_{x \in V}\left(\Lambda_x \right)^{\eta_t(x)}\right]
	= \left(\prod_{x \in V} J_{\scale_L,\scale_R,\lambda_x}\right) \E_{\mu_{\boldsymbol k}}[\psi_t],
	\end{align}
	where  $\psi$ and $\psi_t$ are given in \eqref{eq:psi} and \eqref{eq:psi_t}, respectively,  $\mu_{\boldsymbol k}=\otimes_{x\in V} \nu_{ x, \kappa_x}$ is the probability measure  defined in \eqref{eq:muscale} with  parameters $\boldsymbol \kappa=\{\kappa_x: x \in V\}$, viewed as a probability measure on $\widehat{\mathcal X} $  concentrated on $\widehat{\mathcal Y}$, and
	\begin{align}\label{eq: J}
	J_{\scale_L,\scale_R,\lambda_x}:= \begin{dcases}
	e^{\alpha_x\lambda_x(\bar \scale_x+(\scale_L-\scale_R))} &\text{if}\ \sigma = 0\\
	\left(\frac{1+\sigma \lambda_x (1+\bar \scale_x)}{1+\sigma \lambda_x(1-(\scale_L-\scale_R))} \right)^{\sigma \alpha_x} &\text{if}\ \sigma \in \{-1,1\}\ .	
	\end{dcases}
	\end{align}
\end{theorem}
\begin{remark}[\textsc{conditions \eqref{eq: Lambda_x} \& \eqref{eq: K_x}}]		
	Condition \eqref{eq: Lambda_x} is obtained for
	\begin{align}
	\lambda_x\subset \begin{dcases}
	\left(-\infty,\frac{1}{1+\scale_x}\right]\cup \left[\frac{1}{\scale_x},\infty\right)&\text{if}\ \sigma = -1\\
	\left[-1,\infty\right) &\text{if}\ \sigma = 0\\
	\left(-\infty,-\frac{1}{1+\scale_x}\right]\cup \left[-\frac{1}{2+\scale_x},\infty\right)&\text{if}\ \sigma = 1\ ,
	\end{dcases}
	\end{align}
	while condition \eqref{eq: K_x} for
	\begin{enumerate}[label={\normalfont (\roman*)}]
		\item Case $\scale_L-\scale_R\ge 0$ : \begin{align}
		\frac{1}{\lambda_x}\subset \begin{dcases}
		\left[1,\infty\right)&\text{if}\ \sigma = -1\\
		\left[0,\infty\right) &\text{if}\ \sigma = 0\\
		\left[\scale_L-\scale_R-1,\infty\right)&\text{if}\ \sigma = 1\ ,
		\end{dcases}
		\end{align}
		\item Case $\scale_L-\scale_R\le 0$ :\begin{align}
		\frac{1}{\lambda_x}\subset \begin{dcases}
		(-\infty,1-\scale_L+\scale_R]&\text{if}\ \sigma = -1\\
		(-\infty,-1) &\text{if}\ \sigma = 0\\
		(-\infty,\scale_L-\scale_R-1]&\text{if}\ \sigma = 1.\\
		\end{dcases}
		\end{align}
	\end{enumerate}
\end{remark}

We devote the remaining of this section to the proof of Theorem \ref{theorem joint moment generating function}. To this purpose, let us recall the definition of  $\{w_x: x \in V\}$ and $\{z_{x,\cdot}: x \in V \}$ in \eqref{eq:wx} and \eqref{eq:z}, respectively.
\begin{definition}[\textsc{single-site generating functions}] \label{definition:intertwiners}
	For each choice of $\sigma \in \{-1,0,1\}$, for all $x \in V$ and for all functions $f: \N_0 \to \R$, we define
	\begin{align}\label{eq:single-site_genfunction}
	\left(\varUpsilon_x f\right)(\lambda)\ :=\ \sum_{k=0}^\infty \frac{w_x(k)}{k!} \frac{\left(\frac{\lambda}{1+\sigma \lambda}\right)^k}{z_{x,\lambda}} f(k),
	\end{align}
	\begin{align*}
	(	\varUpsilon_L f)(\lambda)&:= \sum_{k=0}^\infty \frac{(\alpha_L \lambda)^k}{k!} f(k)\, e^{-\alpha_L\lambda}
	\end{align*}
	and
	\begin{align*}
	(	\varUpsilon_R f)(\lambda)&:= \sum_{k=0}^\infty \frac{(\alpha_R \lambda)^k}{k!} f(k)\,e^{-\alpha_R\lambda}\
	\end{align*}
	for all $\lambda \in \R$ such that the above series absolutely converge. 	Moreover, we define \begin{align}\label{eq:intertwiner}
	\varUpsilon:= \varUpsilon_L \otimes \left(\otimes_{x \in V} \varUpsilon_x \right) \otimes \varUpsilon_R\ ,
	\end{align}
	acting on functions $f:\N_0^{N+2}\to \R$.
\end{definition}
\begin{remark}
	If $\lambda \in \Scale$ then, for all $x \in V$ and $f: \N_0 \to \R$, $$	\left(\varUpsilon_x f\right)(\lambda)= \E_{\nu_{x,\lambda}}[f]\ ,$$
	where $\nu_{x,\lambda_x}$ is given in \eqref{eq:explicitnu}.
\end{remark}

As a first step, we investigate the action of the operators $\{\varUpsilon_x: x \in V\}$ on the duality functions.
\begin{lemma}[\textsc{duality and generating functions}]
	For each choice of $\sigma \in \{-1,0,1\}$, for all $\scale \in \Scale$ and for all  $x \in V$,
	\begin{align}\label{eq:single-site_genfunction_classical}
	(\varUpsilon_x)_\eft\, d^\cl_x(\cdot,n)(\lambda)= \frac{\left(1+\frac{\lambda}{1+\sigma \lambda}\right)^n}{z_{x,\lambda}}
	\end{align}
	and
	\begin{align}\label{eq:single-site_genfunction_orthogonal}
	(\varUpsilon_x)_\eft\, d^\oo_{x,\scale}(\cdot,n)(\lambda)  = \frac{\left(1+\frac{\lambda}{1+\sigma \lambda(1+\scale)} \right)^n}{z_{x,\lambda(1+\scale)}}\ .
	\end{align}
	Moreover
	\begin{align*}
	\varUpsilon_\eft D^\oo_{ \scale}(\cdot,\eta)(\boldsymbol \lambda) = e^{-\alpha_L \lambda_L (1+\scale-\scale_L)}	
	\left( \prod_{x \in V} \frac{\left(1+\frac{\lambda_x}{1+\sigma \lambda_x(1+\scale)} \right)^{\eta(x)}}{z_{x,\lambda_x(1+\scale)}}\right)
	e^{-\alpha_R \lambda_R (1+\scale-\scale_R)}\ ,
	\end{align*}
	and, analogously,
	\begin{align}\label{eq:fake_duality+intertwining}
	&\varUpsilon_\eft D^\oo_{\bar{\boldsymbol\scale}}(\cdot,\eta)(\boldsymbol \lambda) = e^{-\alpha_L \lambda_L }	
	\left( \prod_{x \in V} \frac{\left(1+\frac{\lambda_x}{1+\sigma \lambda_x(1+\bar\scale_x)} \right)^{\eta(x)}}{z_{x,\lambda_x(1+\bar \scale_x)}}\right)
	e^{-\alpha_R \lambda_R}\ .	
	\end{align}
\end{lemma}

\begin{remark} In order to guarantee the absolute convergence of the series in the definition of the operators $\varUpsilon$ in Definition \ref{definition:intertwiners}, for the case $\sigma=1$ we have to choose $\lambda$ and $\scale$ such that
	$$\left|\frac{\scale \lambda}{1+ \lambda} \right| < 1\ .$$
\end{remark}

\begin{proof}
	By \eqref{eq:classical-orthogonal_theta=0}, we prove \eqref{eq:single-site_genfunction_orthogonal} from which, by setting $\scale=0$,  \eqref{eq:single-site_genfunction_classical} follows. By definition of $\varUpsilon_x$ in \eqref{eq:single-site_genfunction}, relation \eqref{eq:fake_binomial} and the form of the functions $\{w_x: x \in V\}$ (see \eqref{eq:wx}), we obtain
	\begin{align*}
	(\varUpsilon_x)_\eft\, d^\oo_{x,\scale}(\cdot,n)(\lambda) =&\ \sum_{k=0}^\infty \frac{w_x(k)}{k!} \frac{\left(\frac{\lambda}{1+\sigma \lambda} \right)^k}{z_{x,\lambda}} d^\oo_{x,\scale}(k,n)\\
	=&\ \sum_{\ell=0}^n \binom{n}{\ell} \frac{\left(\frac{\lambda}{1+\sigma \lambda} \right)^\ell}{z_{x,\lambda}} \sum_{k=\ell}^\infty \frac{w_x(k) }{w_x(\ell)(k-\ell)!}\left(\frac{-\scale \lambda}{1+\sigma \lambda} \right)^{k-\ell}\\
	=&\  \sum_{\ell=0}^n \binom{n}{\ell} \frac{\left(\frac{ \lambda}{1+\sigma \lambda} \right)^\ell}{z_{x,\lambda}} F_x(\scale,\lambda,\ell)\ ,
	\end{align*}
	where, as long as $\left|\frac{\scale \lambda}{1+ \lambda} \right| < 1$ if $\sigma = 1$ and for all $\lambda \in \R$ otherwise,
	\begin{align*}
	F_x(\scale,\lambda,\ell)=\begin{dcases}
	\left(\frac{1+\sigma \lambda (1+\scale)}{1+\sigma\lambda} \right)^{-(\sigma \alpha_x +\ell) } &\text{if } \sigma \in \{-1,1\}
	\\
	e^{-\alpha_x \scale \lambda} &\text{if } \sigma =  0\ .
	\end{dcases}
	\end{align*}
\end{proof}

\begin{proof}[Proof of Theorem~\ref{theorem joint moment generating function}]
	We start by proving \eqref{eq:identity_frank time t}. First, by \eqref{eq:fake_duality+intertwining}, the l.h.s.\ in \eqref{eq:identity_frank time t} equals 
	\begin{align*}
	\text{l.h.s.\ in } \eqref{eq:identity_frank time t}&= \E_{\mu_{\bar{\boldsymbol \scale}}}\left[\varUpsilon_\eft D^\oo_{\bar {\boldsymbol \scale}}(\cdot,\eta_t) (\boldsymbol \lambda)\right] e^{\alpha_L\lambda_L+\alpha_R\lambda_R} \left(\prod_{x \in V} z_{x,\lambda_x(1+\bar \scale_x)}\right)
	\\&= \varUpsilon \left(\left(\scale_L-\scale_R \right)^{|\cdot|} \psi_t(\cdot) \right)(\boldsymbol \lambda)\, e^{\alpha_L\lambda_L+\alpha_R\lambda_R}\left(\prod_{x \in V} z_{x,\lambda_x(1+\bar \scale_x)}\right)\ ,
	\end{align*}
where in the second identity we have 	exchanged  $\varUpsilon_\eft$ and the expectation w.r.t.\  $\eta$ -- two operators acting on different variables -- together with  \eqref{eq:psi_t5}.
By the definition of $\varUpsilon$ (cf.\ Definition \ref{definition:intertwiners}) and
\eqref{eq:zero_boundary_psi} (cf.\ \eqref{eq:Y}), we further get 
\begin{align*}
\text{l.h.s.\ in } \eqref{eq:identity_frank time t}	&= \sum_{\xi \in \widehat{\mathcal Y}} \left(\prod_{x \in V}\frac{w_x(\xi(x))}{(\xi(x))!}\frac{\left(\frac{\lambda_x\left(\scale_L-\scale_R \right)}{1+\sigma \lambda_x} \right)^{\xi(x)}}{z_{x,\lambda_x}} \right) \psi_{t}(\xi)\left(\prod_{x \in V} z_{x,\lambda_x(1+\bar \scale_x)}\right)\ ,
	\end{align*}
which, by the definition of $\mu_{\boldsymbol \kappa}$ (cf.\ the statement of the theorem), equals
\begin{align*}
	\text{l.h.s.\ in } \eqref{eq:identity_frank time t}= \left(\prod_{x \in V} \frac{z_{x,\lambda_x(1+\bar \scale_x)}z_{x,\kappa_x}}{z_{x,\lambda_x}} \right)\sum_{\xi \in \widehat{\mathcal X}} \mu_{\boldsymbol \kappa}(\xi)\,  \psi_{t}(\xi)\ .
	\end{align*}
The explicit form of  $\{z_{x,\cdot}: x \in V\}$ given in \eqref{eq:z} yields \eqref{eq:identity_frank time t}. Sending $t\to \infty$ in \eqref{eq:identity_frank time t}, by the uniqueness of the stationary measure, we obtain \eqref{eq:identity_frank}.
\end{proof}

\appendix

\section{Existence and uniqueness of the equilibrium and non-equilibrium stationary measure}\label{appendix:existence_uniqueness}
In this appendix, we treat with  full details the issue of existence and uniqueness of the stationary measure for $\IRW$ and $\SIP$ in equilibrium and non-equilibrium. In what follows we take either $\sigma = 0$ or $\sigma =1$.

We recall that a probability measure $\mu$ on the countable space $\mathcal X$ (endowed with the discrete topology) is the unique stationary measure  for the particle system $\{\eta_t: t \geq 0 \}$ if, for all bounded functions $f : \mathcal X \to \R$ and for all probability measures $\mu'$ on $\mathcal X$,  the following holds:
\begin{equation}
\lim_{t\to \infty} \E_{\mu'}\left[f(\eta_t)\right] = \E_{\mu}\left[f(\eta) \right]\ .
\end{equation}
Out of all probability measures $\mu$ on $\mathcal X$, we say that $\mu$ is \emph{tempered} if it is characterized by the integrals
\begin{align*}
\E_\mu\left[D^\cl(\xi,\eta)\right]\ ,\quad \text{for all}\quad \xi \in \widehat{\mathcal X}\ .
\end{align*}
To the purpose of determining whether a probability measure $\mu$ is tempered or not, we adopt the following strategy. First, we recall that the  functions $\{D^\cl(\xi,\cdot): \xi \in \widehat{\mathcal X} \}$ are weighted products of factorial moments of the variables $\{\eta(x): x \in V \}$ (see Proposition \ref{proposition:classical_duality}). Then, we express these weighted factorial moments in terms of moments. We conclude by means of a multidimensional Carleman's condition.

By following the aforementioned ideas,  we provide in the following lemma a sufficient condition for a measure to be tempered.
\begin{lemma}\label{lemma:tempered}
	Let $\mu$ be a probability measure on $\mathcal X$. If there exists $\scale \in \Scale=[0,\infty)$ such that
	\begin{align}\label{eq:upperbound_tempered}
	\E_\mu\left[D^\cl(\xi,\eta) \right] \leq \scale^{|\xi|}
	\end{align}
	for all $\xi \in \widehat{\mathcal X}$, then $\mu$ is tempered.
\end{lemma}
\begin{proof}
	Let us start by expressing the moments of $\eta(x)$ in terms of single-site classical duality functions in \eqref{eq:single-site_classical}: for all $x \in V$ and for all $k, n \in \N_0$,
	\begin{align*}
	n^k = \sum_{\ell = 0}^k  \stirling{k}{\ell}\,  d^\cl_x(\ell,n)\, w_x(\ell)\ ,
	\end{align*}
	where $\stirling{k}{\ell}$ denotes the Stirling number of the second kind given by
	\begin{equation}
	\stirling{k}{\ell} = \frac{1}{\ell!} \sum_{j=0}^\ell (-1)^{\ell-j}\, \binom{\ell}{j}\, j^k\ .	
	\end{equation} In view of \eqref{eq:upperbound_tempered}, we obtain
	\begin{align*}
	\E_\mu\left[(\eta(x))^k \right]	&= \sum_{\ell = 0}^k  \stirling{k}{\ell}\, \E_\mu\left[D^\cl(\ell\delta_x,\eta) \right] w_x(\ell)\\
	&\leq \sum_{\ell=0}^k \frac{w_x(\ell)}{\ell!}\, \E_\mu\left[ D^\cl(\ell\delta_x,\eta)\right] \sum_{j=0}^\ell  \binom{\ell}{j}\, j^k\\
	&\leq k^k\sum_{\ell=0}^k \frac{ (2\scale)^\ell}{\ell!} w_x(\ell)\ .
	\end{align*}
	By recalling the definition of $w_x(\ell)$ in \eqref{eq:wx}, in both cases with $\sigma = 0$ and $\sigma = 1$, we get
	\begin{align}\label{eq:upperbound_carleman}
	\E_\mu\left[(\eta(x))^k \right]\ \leq\  (a_xk)^k\ ,
	\end{align}
	for all $k \in \N$,
	with $a_x =  (1+2\scale \alpha_x)$ for $\sigma =0$ and $a_x = \lfloor \alpha_x\rfloor! (1+2\scale)^{\lfloor \alpha_x\rfloor+1}$ for $\sigma = 1$.
	Therefore, if $m_x(k):= \E_\mu\left[(\eta(x))^k \right]$, \eqref{eq:upperbound_carleman} yields
	\begin{align*}
	\sum_{k=1}^\infty m_x(2k)^{-\frac{1}{2k}}\geq \frac{1}{a_x} \sum_{k=1}^\infty  \frac{1}{2k} = \infty\ .
	\end{align*}
	Because the above condition holds for all $x \in V$, the multidimensional Carleman condition (see e.g.\ \cite[Theorem 14.19]{schmudgen_moment_2017}) applies. Hence,  $\mu$ is completely characterized by the moments $\{m_x(k): x \in V, k \in \N \}$ and, in turn, is tempered.	
\end{proof}

Now, by means of duality, we observe that, for all $\eta \in \mathcal X$ and $\xi \in \widehat{\mathcal X}$ with $|\xi|=k$,
\begin{align}\label{eq:convergence_classical}
\nonumber
\lim_{t\to \infty} \E_\eta\left[D^\cl(\xi,\eta_t) \right] &= \lim_{t\to \infty} \widehat \E_\xi\left[D^\cl(\xi_t,\eta) \right]\\
&= \sum_{\ell=0}^k \scale_L^\ell\, \scale_R^{k-\ell}\, \widehat \Pr_\xi\left(\xi_\infty = \ell \delta_L + (k-\ell)\delta_R \right)\ .
\end{align}
We note that the expression above does not depend on $\eta \in \mathcal X$ and, moreover,
\begin{align*}
\lim_{t\to \infty} \E_\eta\left[D^\cl(\xi,\eta_t) \right]\leq (\scale_L\vee \scale_R)^{|\xi|}
\end{align*}
for all $\xi \in \widehat{\mathcal X}$.
Therefore, by Lemma \ref{lemma:tempered}, there exists a unique probability measure $\mu_\star$ on $\mathcal X$ such that
\begin{align*}
\E_{\mu_\star}\left[D^\cl(\xi,\eta) \right] = \sum_{\ell=0}^{|\xi|} \scale_L^\ell\, \scale_R^{|\xi|-\ell}\, \widehat \Pr_\xi\left(\xi_\infty = \ell \delta_L + (|\xi|-\ell)\delta_R \right)\ .
\end{align*}
Furthermore, because the convergence in \eqref{eq:convergence_classical} for all $\xi \in \widehat{\mathcal X}$ implies convergence of all marginal moments and because the limiting measure is uniquely characterized by these limiting moments, then, for all $f : \mathcal X \to \R$ bounded and for all $\eta \in \mathcal X$, we have
\begin{align}\label{eq:conv1}
\lim_{t\to \infty} \E_\eta\left[f(\eta_t) \right] = \E_{\mu_\star}\left[f(\eta) \right]\ .
\end{align}
By dominated convergence, \eqref{eq:conv1} yields, for all probability measures $\mu$ on $\mathcal X$ and $f: \mathcal X \to \R$,
\begin{align*}
\lim_{t\to \infty} \E_\mu\left[f(\eta_t) \right] =  \E_\mu\left[\lim_{t\to \infty}\E_\eta\left[f(\eta_t)\right] \right] = \E_{\mu_\star}\left[f(\eta) \right]\ ,
\end{align*}
i.e.\ $\mu_\star$ is the unique stationary measure of the process $\{\eta_t: t\geq 0\}$.	
\qed

\paragraph{Acknowledgments.}
The authors would like to thank Gioia Carinci and Cristian Giardin\`{a} for useful discussions. F.R. and S.F. thank Jean-Ren\'{e} Chazottes for a stay at CPHT (Institut Polytechnique de Paris), in the realm of Chaire d'Alembert (Paris-Saclay University), where part of this work was performed. S.F. acknowledges Simona Villa for her  support in creating the picture. S.F.\ acknowledges financial support from NWO via the 	grant
TOP1.17.019.  F.S.\ acknowledges financial support from the European Union's Horizon 2020 research and innovation programme under the Marie-Sk\l{}odowska-Curie grant agreement No.754411.

\bibliographystyle{acm}

\end{document}